\documentclass{article} 
\usepackage{slashed}
\usepackage[a4paper]{geometry}
\usepackage{amsthm}
\usepackage{mathrsfs} 
\usepackage{amssymb}
\usepackage{amsopn}
\usepackage{amsmath}

\usepackage{hyperref}
\usepackage{enumitem}

\DeclareMathOperator{\End}{End}
\DeclareMathOperator{\Hom}{Hom}

\DeclareMathOperator{\tr}{tr}
\DeclareMathOperator{\id}{Id}

\def\cF{\mathscr{F}}
\def\cH{\mathscr{H}}

\def\cJ{\mathscr{J}}

\def\cO{\mathscr{O}}

\def\fg{\mathfrak{g}}

\swapnumbers
\newtheoremstyle{exps}{\topsep}{\topsep}{}{0pt}{\bfseries}{.}{0pt}{}

\newtheorem*{thm*}{Theorem}
\newtheorem*{prop*}{Proposition}
\newtheorem{ques*}{Question}
\newtheorem*{lem*}{Lemma}
\newtheorem*{conj*}{Conjecture}
\newtheorem*{cor*}{Corollary}

\newtheorem*{notn*}{Notation}
\newtheorem*{rem*}{Remark}
\newtheorem{thm}[subsection]{Theorem}

\newtheorem{ques}[subsection]{Question}

\theoremstyle{definition}
\newtheorem*{defn*}{Definition}
\newtheorem*{exer*}{Exercise}
\newtheorem*{ex*}{Example}

\newtheorem{defn}[subsection]{Definition}
\newtheorem*{problem*}{Problem}
\newtheorem{rem}[subsection]{Remark}
\newtheorem{nolabel}[subsection]{}

\theoremstyle{exps}

\numberwithin{equation}{subsection}
\setenumerate[1]{label={\alph*)}} 

\usepackage{latexsym}
\def\vac{\mathbb{1}}
\def\ds{\slashed{D}}
\def\cJ{\mathcal{J}}
\def\fn{\mathfrak{n}}
\def\fh{\mathfrak{h}}
\DeclareMathOperator{\ch}{ch}
\begin{document}
\title{Recent advances and open questions on the susy structure of the chiral de Rham Complex}
\author{Reimundo Heluani \thanks{IMPA, Rio de Janeiro}}
\date{}
\maketitle
\section{Introduction} \label{sec:intro}
\begin{nolabel}\label{no:sheaf}
To each smooth $n$-manifold $M$ one can attach \cite{malikov} a sheaf of super vertex algebras $\Omega_M^{ch}$, called the \emph{chiral de Rham complex of $M$.}. Locally, it consists of $n$ copies of the $bc-\beta\gamma$ system
\[ b_i(z) \cdot c^j(w) \sim \frac{\delta_{i}^j}{z-w}, \qquad \beta_i(z) \cdot \gamma^j(w) \sim \frac{\delta_i^j}{z-w}, \qquad 1 \leq  i,j \leq n, \]
where $b_i, c^i$ are odd fields and $\beta_i, \gamma^i$ are even. 
This vertex algebra carries an odd derivation $\ds$ such that $\ds^2$ equals the translation operator. Locally it is given by
\[ \ds \gamma^i = c^i, \qquad \ds b_i = \beta_i. \]

If $n=2m$ and $M$ is a holomorphic manifold, there exists a holomorphic version $\Omega_M^{ch,hol}$. Locally one has 
\[ \Omega_M^{ch} \simeq \Omega_M^{ch,hol} \otimes \overline{\Omega}_M^{ch,hol}, \]
where $\overline{\Omega}_M^{ch,hol}$ is another copy of the holomorphic sheaf (the anti-holomorphic sector). We will call the corresponding embeddings $\Omega_M^{ch,hol} \hookrightarrow \Omega_M^{ch}$ and $\overline{\Omega}_M^{ch,hol} \hookrightarrow \Omega^{ch}_M$ the \emph{naive embeddings}.

The sheaf $\Omega_M^{ch}$ (resp. $\Omega_M^{ch,hol}$)  depends only on the differentiable (resp. holomorphic) structure of $M$. In particular they do not depend on any metric structure on $M$. 
\end{nolabel}
\begin{nolabel}
When $M$ has special holonomy $\Omega^{ch}_M$ admits two commuting embeddings of certain superconformal extensions of the $N=1$ or Neveu-Schwarz algebra of central charge $c = \tfrac{3n}{2}$. 

A particular case is when $n=2m$ and $M$ is a Calabi-Yau $m$-fold, that is, $M$ has holonomy $SU(m)$. In this case the corresponding supersymmetric extension is the $N=2$ superconformal algebra\footnote{Actually the supersymmetry algebra is an extension by two fields of conformal weight $m/2$ studied by Odake \cite{odake}. We will restrict our attention to the $N=2$ subalgebra.}. There are two ways of understanding the two commuting embeddings of this $N=2$ vertex algebra. Both of these rely in identifying the local generating sections of $\Omega_M^{ch}$ with geometric tensors on $M$. 

For a local coordinate system $\left\{ x^i \right\}_{i=1}^n$ on $M$, we may identify (see \eqref{eq:change1}-\eqref{eq:change2} for a precise statement) 
\begin{equation} \label{eq:assign1}  b_i \leftrightarrow \frac{\partial}{\partial x^i}, \qquad c^i \leftrightarrow dx^i, \qquad \gamma^i \leftrightarrow x^i. \end{equation}
We obtain in this manner embeddings of sheaves of vector spaces
\begin{equation}\label{eq:embedding1-intro} TM \hookrightarrow \Omega_M^{ch} \hookleftarrow T^*M. \end{equation}
To the identity endomorphism of $TM$, locally written as $\sum dx^i \cdot \partial_{x^i}$ one may associate the local section 
\[ J = \sum_{i=1}^n c^i b_i, \]
of $\Omega_M^{ch}$. It turns out that when $M$ is orientable these local expressions glue to give a globally defined section $J \in C^{\infty}(M, \Omega_M^{ch})$. This section $J$ and its \emph{superpartner} $\ds J$ generate one copy of the $N=2$ vertex algebra of central charge $c=3n$. 

If $n=2m$ and $M$ is holomorphic, the above embeddings \eqref{eq:embedding1-intro} are compatible with the complex structures, in the sense that locally, on the coordinate chart with holomorphic coordinates $\left\{ x^\alpha \right\}_{\alpha=1}^m$ and anti-holomorphic coordinates $\left\{ x^{\bar\alpha} \right\}$ 
\[ T_{1,0}M \oplus T_{0,1} M \hookrightarrow \Omega^{ch} \simeq \Omega^{ch,hol} \otimes \overline{\Omega}^{ch,hol} \hookleftarrow T^*_{1,0}M \oplus T^*_{0,1}M. \]
These embeddings generate the naive embeddings mentioned in \ref{no:sheaf}. 

The identity endomorphism of the holomorphic tangent bundle of $M$ and the identity endomorphism of the anti-holomorphic tangent bundle of $M$ give rise to the local sections 
\[ J^{hol} = \sum_\alpha c^\alpha b_\alpha, \qquad \overline{J}^{hol} = \sum_{\bar\alpha} c^{\bar\alpha} b_{\bar\alpha},\]
of $\Omega^{ch,hol}$ and $\overline{\Omega}^{ch,hol}$ respectively. When $M$ is Calabi-Yau, that is $M$ is \emph{holomorphically orientable}, then each of these sections and their superpartners $\ds J^{hol}$, $\ds \overline{J}^{hol}$ produce global sections of $\Omega_{M}^{ch,hol}$ and $\overline{\Omega}_M^{ch, hol}$. Under the naive embeddings these generate the two commuting copies of $N=2$. 
\end{nolabel}
\begin{nolabel}\label{no:non-naive}
A disadvantage of the above approach is that it relies on special coordinates (holomorphic coordinates) to obtain the commuting sectors on $\Omega_M^{ch}$. In particular this will only work for holomorphic manifolds. If we want to obtain two commuting superconformal structures on more general manifolds then we need to identify $TM$ and $T^*M$ in a different way inside of $\Omega_M^{ch}$. This will require a metric on $M$. Let $(M, g)$ be a Riemannian manifold\footnote{Any signature will work in fact.}. If we denote by $g_{ij}$ and $g^{ij}$ the coordinate components of the metric and its inverse, then instead of 
using \eqref{eq:assign1} we  may identify
\begin{equation} \label{eq:assign2} b_i \pm \sum_{i} g_{ij} c^j \leftrightarrow \frac{\partial}{\partial x^i}, \qquad c^i \pm \sum_i g^{ij} b_j \leftrightarrow dx^i, \end{equation}
we obtain now two different embeddings like in \eqref{eq:embedding1-intro}. 

If $n=2m$ and $M$ is a holomorphic manifold, and $g$ is a Kähler metric, the identity endomorphism of the holomorphic tangent bundle now gives rise to two different local sections 
\[ J^\pm = \sum_\alpha \left(c^\alpha \pm \sum_{\bar\beta} g^{\alpha \bar\beta} b_{\bar\beta} \right) \cdot \left( b_\alpha \pm \sum_{\bar\beta} g_{\alpha \bar\beta} c^{\bar\beta} \right). \]
It turns out \cite{heluani8} that when $(M,g)$ is Ricci-Flat these two sections are globally defined and together with their superpartners $\ds J^\pm$ they generate two commuting copies of the $N=2$ superconformal algebra of central charge $c=\frac{3n}{2}$. We would have obtained the same two sections had we started with the identity endomorphism of the anti-holomorphic tangent bundle. 
\end{nolabel}
\begin{nolabel}\label{no:general-holonomy}
One advantage of the approach described above is that this can be applied to any manifold with a metric, without the need of having special coordinate systems. This in particular makes it possible to study supersymmetry in other special holonomy cases. Using the embedding \eqref{eq:assign2} instead of the naive embedding, we obtain two global sections of $\Omega_M^{ch}$ for each globally defined differential form on $M$. When we consider only those that are paralell with respect to the Levi-Civita connection they generate the above mentioned supersymmetric extensions of $N=1$. The known results and the open questions are reviewed in detail in section \ref{sec:susy-cases}. 
\end{nolabel}

\begin{nolabel}\label{no:modular}
Interest in the supersymmetry of $\Omega_M^{ch}$ has been renewed due to the result in \cite{eguchi} connecting the elliptic genus of a $K3$ surface with Mathieu's $M_{24}$ group. The authors decompose the elliptic genus $Z(\tau, \alpha)$ of a $K3$ surface as linear combination of irreducible characters of the $N=4$ superconformal algebra of central charge $c=6$ and found that the generating series of multiplicities is a mock modular form whose coefficients in its $q$-expansion are dimensions of representations of $M_{24}$. 

Since the elliptic genus of a Calabi-Yau manifold is the graded dimension of the cohomology of the chiral de Rham complex, it is natural to ask if one can apply the same procedure to other special holonomy manifolds, decomposing $\Omega_{M}^{ch}$ (resp. its graded dimension) as sums of irreducible modules of the corresponding supersymmetry algebra (resp. its irreducible characters). 

An advantage of the above described approach of working with the $C^\infty$ version of $\Omega_M^{ch}$ is that it makes it possible to formalize expressions for the elliptic genus as
\[ Z(\tau, \alpha) = \tr_{C^{\infty}(M, \Omega_M^{ch})} (-1)^F y^{J_0^+} q^{L_0^+ - c/24} \bar{y}^{J_0^-} \bar{q}^{L^-_0 - c/24}, \]
appearing in the physics literature. A disadvantage is that the spaces of global smooth sections, even of a prescribed conformal weight and charge, are infinite dimensional. One can correct this by taking a BRST cohomology with respect to the ``minus'' or ``left-moving'' sector, to obtain a superconformal vertex algebra with convergent character. From this perspective, we view $\Omega_{M}^{ch}$ as a Dolbeaut resolution of $\Omega_M^{ch,hol}$. This is known as a topological twist and we discuss below the proposal in \cite{boer} for $G_2$ manifolds applied in the context of the chiral de Rham complex of $M$. We also discuss this topological twist in the case of $Spin_7$ manifolds, as a possible way to rigorously analyze the results of \cite{benjamin-harrison} from the point of view of $\Omega_M^{ch}$. 
\end{nolabel}
\begin{nolabel}\label{no:mock}
The elliptic genus $Z(\tau)$ of complex manifolds $M$, or (conjecturally) more generally the graded dimension of the cohomology of $\Omega_M^{ch}$ has modular properties. Knowing that in the special holonomy cases we obtain explicit subalgebras acting in $\Omega_M^{ch}$ it is natural to try to express $\Omega_M^{ch}$ as a sum of irreducible modules over these subalgebras, os simply expand $Z(\tau)$ as linear combination of irreducible characters of these subalgebras. The generating series of the multiplicity spaces is expected to have mock modular properties. 

In section \ref{sec:hamiltonian} below we discuss this situation from the point of view of Hamiltonian reduction. It turns out that all of the supersymmetric extensions of $N=1$ appearing inside of $\Omega_M^{ch}$ for special holonomy manifolds are quantum Hamiltonian reductions of affine Lie superalgebras at either the \emph{superprincipal} or minimal nilpotent (see \ref{sec:hamiltonian}). Our understanding of the representation theory of these algebras has advanced dramatically in these last few years, including the mock-modularity of their irreducible representations (see \cite{kacwakimoto-superconformal} and references therein). These new techniques can be applied to obtain the conjectural characters of \cite{benjamin-harrison} in the $Spin_7$ case and perhaps apply the same techniques to assign a mock modular form to each $G_2$ manifold. The connection between the quantum Hamiltonian reduction and the chiral de Rham complex $\Omega_M^{ch}$ remains a mystery to this day. 
\end{nolabel}
\begin{nolabel}\label{no:purpose}
The purpose of this article is to clarify the different embeddings of the superconformal algebras in the chiral de Rham complex $\Omega_M^{ch}$ available in the literature. We explicitly discuss the differences between the holomorphic--anti-holomorphic sectors vs the left--right sectors. This boils down to interpreting the chiral de Rham complex locally as a free ghost system vs a free Boson-Fermion system (see \ref{no:boson-fermion-vs-ghosts}). We describe the known results and conjectures regarding the existence of these supersymmetric algebras. 

We leave several open questions. All of them towards applying the following guideline. 1) construct two commuting superconformal structures on the $C^\infty$ chiral de Rham complex. 2) Take global sections (as opposed to sheaf cohomology) and then BRST cohomology of the left moving sector. 3) Decompose this cohomology with respect to the right moving superconformal structure. 
\end{nolabel}

\section{Vertex algebras} \label{sec:examples}
In this section we collect some notation and examples of vertex algebras with  $N=1$ supersymmetric structure. These algebras where introduced by Barron in \cite{barron1,barron2,barron3,barron4}, studied by Kac in \cite{kac:vertex} and by Kac and the author in \cite{heluani3}. To keep the notation simple we will not make use of the superfield formalism and instead treat these algebras as ordinary vertex algebras endowed with an extra odd derivation. 

We will avoid the prefix \emph{super} whenever possible so unless otherwise noticed, all of our vertex algebras are indeed super vertex algebras. For two homogeneous elements $a, b$ in a vector (super) space $V$ we will denote by $(-1)^{ab}$ the number $-1$ if both $a$ and $b$ are odd, $+1$ otherwise. 

This is not meant to be an introduction to vertex algebras and we assume that the reader is familiar with the literature on the subject as very good textbooks are available. We here just collect notation in order to quickly introduce the main examples that will be used below. 

\begin{nolabel}\label{no:fields}
For a vector space $V$ we denote 
\[ \cF(V) = \Hom \left( V, V( (z)) \right), \]
and call its elements \emph{fields on $V$}. For $a(z) \in \cF(V)$ we will write 
\[ a(z) = \sum_{n \in \mathbb{Z}} z^{-1-n} a_{(n)}, \] 
where $a_{(n)} \in \End(V)$ satisfy that for each $b \in V$ there exists an $m \in \mathbb{Z}$ such that $a_{(n)}b = 0$ for all $n \geq m$.
\end{nolabel}
\begin{nolabel}\label{no:vertex}
Recall that a vertex algebra consists of a vector space $V$ together with an even vector $\vac \in V$, an even endomorphism $\partial \in \End(V)$ and a bilinear map \[ V \otimes V \rightarrow V ( (z)), \qquad a \otimes b \mapsto a(z)b = \sum_{n \in \mathbb{Z}} z^{-1-n} a_{(n)}b. \] 
Dualizing the second factor in this map this is equivalent to giving a map 
\[ V \mapsto \cF(V), \qquad a \mapsto a(z) = \sum_{n \in \mathbb{Z}} z^{-1-n} a_{(n)}. \] 
 We will call $a(z)$ the field associated to $a$.
This data is subject to the following axioms:
\begin{itemize}
\item $\vac(z) a = a$, $a(z) \vac = e^{z \partial} a$. 
\item $[\partial, a(z)] = \partial_z a(z)$. 
\item $(z-w)^n a(z) b(w) = (-1)^{ab} (z-w)^n b(w) a(z)$, for all $a, b \in V$ and $n \gg 0$. 
\end{itemize}
\end{nolabel}
\begin{nolabel}\label{no:OPE}
Given two vectors on $a,b$ on a vertex algebra $V$ we will denote
\[ a(z) \cdot b(w) \sim \sum_{j \geq 0} \frac{(a_{(j)}b) (w) }{(z-w)^{j+1}}, \]
and call this the OPE of $a$ and $b$. Note that the sum on the RHS is finite. If $a_{(j)}b = \alpha \vac$ for some $\alpha \in \mathbb{C}$ we will omit $\vac(w)$ in the above notation and implicitly write $\alpha$ for $\alpha \id_{V}$. 

Similarly we will denote by $a \cdot b = a_{(-1)}b$. This is called the \emph{normally ordered product} of $a$ and $b$. It is in general neither a commutative nor an associative bilinear structure. It may happen however that for three specific vectors $a,b,c \in V$ we may have $(a\cdot b) \cdot c = a \cdot (b \cdot c)$. In these cases we will not write the parenthesis. 
\end{nolabel}
\begin{nolabel}\label{no:generators}
Let $V$ be a vertex algebra and let $S \subset V$ be an ordered finite set of vectors in $V$ satisfying 
\begin{equation} \label{eq:ordered1} V = \text{span} \left\{ a^1_{-n_1} a^2_{-n_2} \cdots a^k_{-n_k}\vac, S \ni a^i \leq a^{i+1}, \quad  a^i = a^{i+1} \Rightarrow n_i \geq n_{i+1}, \quad n_i \geq 1 \right\} \end{equation}

We will say that $S$ generates $V$. Since all operations are bilinear we may replace $S$ by a finite dimensional vector space of $V$. 

If a vertex algebra has a finite set of generators $S$ then the OPE of any two vectors from $V$ can be uniquely determined from the OPEs of the vectors from $S$. Below we will list examples of vertex algebras admitting a finite set of generators. For this we simply say what the generators are, and write their OPE. Often times we omit the ordering in $S$ implicitly assuming that any ordering leads to a generating set.  

The notion of generators, \emph{strong generators}, and PBW-like theorems can be formalized and generalized in detail. We point the readers to the classical literature on the subject. When the OPEs or positive products $a_{(j)}b$ for vectors $a,b \in S$ and $j \geq 0$ is a linear combination of vectors in $S$ or some derivative of them (that is $k=1$ above) the $\mathbb{C}[\partial]$-submodule of $V$ generated by $S$ is called a \emph{Lie conformal algebra}, a \emph{vertex Lie algebra} or a $Lie^*$ algebra. 
\end{nolabel}
\begin{nolabel}\label{no:replace-partial}
Notice that we can replace \eqref{eq:ordered1} by 
\[ \label{eq:ordered2} V = \text{span} \left\{ (\partial^{n_1} a^1) \Bigl(  (\partial^{n_2} a^2) \Bigl( \cdots (\partial^{n_k} a^k)\Bigr) \cdots  \Bigr), S \ni a^i \leq a^{i+1}, \quad  a^i = a^{i+1} \Rightarrow n_i \geq n_{i+1}, \quad n_i \geq 1 \right\}.\]
\end{nolabel}
\begin{nolabel}\label{no:vg}
Let $\fg$ be a finite dimensional Lie algebra with an invariant non-degenerate bilinear form $\langle, \rangle$. The affine Kac-Moody vertex algebra is generated by $\fg$ with the OPE given by 
\[ a(z) \cdot b(w) \sim \frac{[a,b](w)}{z-w} + \frac{\langle a,b\rangle}{(z-w)^2}. \]
When $\fg$ is simple, it is customary to write $\langle, \rangle_0$ for the invariant bilinear form such that the longest root $\theta$ satisfies $\langle \theta,\theta \rangle_0 = 2$, and when $\langle, \rangle = k \langle, \rangle_0$ the corresponding vertex algebra is called the \emph{affine Kac-Moody} vertex algebra of level $k$. As a vector space it is given by the Verma $\hat{\fg}$-module $V(k \Lambda_0)$. 
\end{nolabel}
\begin{nolabel}\label{no:free-fermions}
Let $V$ be a vector (super) space with a (super) symmetric bilinear form $(,)$. The algebra of \emph{free Fermions} $F(V)$ is generated by vectors $v \in V$ with reversed parity and OPEs
\[ v(z) \cdot v'(w) \sim \frac{(v,v')}{(z-w)}. \] 
\end{nolabel}
\begin{nolabel}\label{no:symplectic-bosons}
Let $V$ be a vector (super) space with a (super) symplectic bilinear form $(,)$. The algebra of \emph{symplectic Bosons} is generated by vectors $v \in V$ (same parity) and OPEs
\[ v(z) \cdot v'(w) \sim \frac{(v,v')}{(z-w)}. \] 
\end{nolabel}
\begin{nolabel}\label{no:Virasoro}
The Virasoro vertex algebra of central charge $c$ is generated by one vector $L$ with OPE
\begin{equation} \label{eq:virasoro1}  L(z)\cdot L(w) \sim \frac{\partial_w L(w)}{(z-w)} + \frac{2 L(w)}{(z-w)^2} + \frac{c/2}{(z-w)^4}. \end{equation}

A vertex algebra $V$ is called conformal if it has a vector $L \in V$ such that $L(z)$ satisfies \eqref{eq:virasoro1} and moreover expanding 
\begin{equation} \label{eq:virasoro2}  L(z) = \sum_{n \in \mathbb{Z}}z^{-2-n} L_n, \end{equation}
we have $L_{-1} = \partial$ and $L_0$ is diagonalizable with spectrum bounded below (we will work with spectrum being non-negative half-integer or integer in this article). The number $c$ will be called the central charge of the vertex algebra $V$. 

Every conformal vertex algebra is a representation of the Virasoro Lie algebra of central charge $c$. Primitive vectors of this representations are called \emph{primary} vectors and their eigenvalues for $L_0$ are called their conformal weight. That is $a \in V$ is a primary vector of conformal weight $\Delta$ if 
\begin{equation} \label{eq:primary} L(z) \cdot a(w) \sim \frac{\partial_w a(w)}{(z-w)} + \frac{\Delta a(w)}{(z-w)^2}. \end{equation} 
If there are higher order poles in the OPE the vector $a$ is said to have conformal weight $\Delta$ but it is not primary. 
\end{nolabel}
\begin{nolabel}\label{no:n=1}
The Neveu-Schwarz or $N=1$ vertex algebra of central charge $c$ is a super extension of the Virasoro vertex algebra of the same central charge. With a vector $L$ as in \ref{no:Virasoro} and another odd generator $G$, primary of conformal weight $3/2$. The remaining OPE is given by
\begin{equation} \label{eq:n=1} G(z) \cdot G(w) \sim \frac{2L(w)}{z-w} + \frac{c}{(z-w)^2}. \end{equation}
If we expand 
\begin{equation} \label{eq:g-expansion}
G(z) = \sum_{n \in \mathbb{Z}} G_{(n)}z^{-1-n} = \sum_{n \in 1/2 + \mathbb{Z}} G_{n} z^{-3/2 - n}, 
\end{equation} 
Then the operator 
\begin{equation} \label{eq:ds-def} \ds := G_{(0)} = G_{-1/2}, \end{equation} is an odd endomorphism of $V$ satisfying 
\begin{equation} \label{eq:ds2} \ds^2 = \partial. \end{equation} 

We will say that a vertex algebra is an $N=1$ superconformal vertex algebra of central charge $c$ if it is conformal of central charge $c$ and it has an odd vector $G \in V$, primary of conformal weight $3/2$ such that $G(z)$ satisfies \eqref{eq:n=1} and moreover defining $\ds$ by \eqref{eq:ds-def} we have 
\begin{equation} \label{eq:super-translation}
[\ds, a(z)] = (\ds a)(z). 
\end{equation} 
A vector $a \in V$ will be called primary of conformal weight $\Delta$ if it has conformal weight $\Delta$ and it is a primitive vector for the representation of the $N=1$ superconformal Lie algebra. That is in addition of \eqref{eq:primary} we require
\begin{equation} \label{eq:primary2}
G(z) \cdot a(w) \sim \frac{(\ds a )(w)}{(z-w)}.
\end{equation}
\end{nolabel}
\begin{defn}
A supersymmetric vertex algebra consist of a vertex algebra $V$ together with an odd endomorphism $\ds$ satisfying $\ds \vac = 0$, \eqref{eq:ds2} and \eqref{eq:super-translation}. For a supersymmetric vertex algebra $V$ we will call $\ds a$ the \emph{superpartner} of $a$. 

In particular, any $N=1$ superconformal vertex algebra is a supersymmetric vertex algebra. The converse is not true however. 
\label{defn:susy1}
\end{defn}
\begin{nolabel}\label{no:susy-generators}
Given a supersymmetric vertex algebra $V$ we will say that a finite ordered set $S \subset V$ generates $V$ if 
\[ \label{eq:ordered3} V = \text{span} \left\{ (\ds^{n_1} a^1) \Bigl(  (\ds^{n_2} a^2) \Bigl( \cdots (\ds^{n_k} a^k)\Bigr) \cdots  \Bigr), S \ni a^i \leq a^{i+1}, \quad  a^i = a^{i+1} \Rightarrow n_i \geq n_{i+1}, \quad n_i \geq 1 \right\}.\]
Unfortunately, even when a supersymmetric vertex algebra admits a set of generators, it is not enough to know the OPEs between the generators, we also need to know the OPE of these generators with their superpartners\footnote{This is resolved by the use of superfields and their super OPE or Lambda-bracket as in \cite{heluani3}.}.
\end{nolabel}
\begin{nolabel}\label{no:n=1-generated}
The $N=1$ vertex algebra of \ref{no:n=1} is generated, as a supersymmetric vertex algebra by the odd vector $G$. Its superpartner is $2 L = \ds G$. 
\end{nolabel}
\begin{nolabel}\label{no:vgsuper}
Let $\fg$ be a finite dimensional Lie algebra with a non-degenerate invariant bilinear form $\langle, \rangle$. We have a supersymmetric vertex algebra $V(\fg_{super})$ which is generated by vectors $\bar{a} \in \fg$ with reversed parity and their superpartners $a:=\ds \bar{a}$, with OPEs given by
\[ \bar{a}(z) \cdot \bar{b}(w) \sim \frac{\langle a,b \rangle}{z-w}, \qquad \ds \bar{a}(z) \cdot \bar{b}(w) \sim \frac{\overline{[a,b]}(w)}{z-w}. \]
It follows that the OPE between the fields $a := \ds \bar{a}$ and  $b:= \ds \bar{b}$ is given by 
\[ a(z) \cdot b(w) \sim \frac{[a,b]}{z-w}+ \frac{\langle a,b \rangle}{(z-w)^2}. \]
This vertex algebra is the same as the equally named $V(\fg_{super})$ in \cite{kac:vertex}. When $\fg$ is simple, normalizing the form $\langle, \rangle = k \langle,\rangle_0$ as in \ref{no:vg} with $k \neq -h^\vee$, that is $k$ is not critical, this algebra is $N=1$ superconformal by the Kac-Todorov construction \cite{kactodorov}. Notice that even at the critical level this is a supersymmetric vertex algebra. As such it is simply generated by the vectors $\bar{a} \in \fg$. When $k \neq -h^\vee$ these vectors are primary of conformal weight $1/2$. 
\end{nolabel}
\begin{nolabel}\label{no:n=2}
The $N=2$ vertex algebra is a superconformal vertex algebra of central charge $c$ generated by an even vector $J$, primary of conformal weight $1$, in addition to the odd $N=1$ vector $G$ as in \ref{no:n=1}. The remaining OPEs are 
\begin{equation} \label{eq:n=2-j-ope} J(z)\cdot J(w) \sim \frac{c/3}{(z-w)^2}, \qquad (\ds J)(z) \cdot J(w) \sim \frac{G(w)}{z-w}. \end{equation} 

As an ordinary vertex algebra, it has $4$ generators, $G$, primary of conformal weight $3/2$, its superpartner $L = \tfrac{1}{2} \ds G$ which is a Virasoro vector hence of conformal weight $2$,  $J$ of conformal weight $1$ and its superpartner $\ds J$, primary of conformal weight $3/2$. It is customary to define
\[ G = G^+ + G^-, \qquad \ds J = G^- - G^+. \]
The vectors $G^\pm$ are primary of conformal weight $3/2$ with respect to the Virasoro element $L$ and have charge $\pm 1$ with respect to the $U(1)$ current $J$, that is 
\[ J(z) \cdot G^\pm(w) \sim \pm \frac{G^\pm(w)}{z-w}. \]
Decomposing 
\begin{equation} \label{eq:j-decomp}
J(z) = \sum_{n \in \mathbb{Z}} z^{-1-n} J_n,
\end{equation}
its zero mode $J_0$ is a semisimple endomorphism of the vertex algebra $V$ with integer eigenvalues. An eigenvector for $J_0$ with eigenvalue $m$ will be called a vector of \emph{charge} $m$, thus, $G^\pm$ has charge $\pm 1$.  

The $N=2$ supersymmetric algebra has an automorphism $\sigma$ preserving the supersymmetric structure. It leaves invariant $G$ and sends $J$ to $-J$. As a usual vertex algebra, this automorphism exchanges therefore $G^+ \leftrightarrow G^-$. 
\end{nolabel}
\begin{nolabel}\label{no:topological-twist1}
The $N=2$ vertex algebra admits another set of generators as a vertex algebra but not as a supersymmetric vertex algebra, in such a way that the Virasoro vector has central charge $0$. This is usually called a \emph{topological twist}. Indeed the vector $T := L + \frac{1}{2} \partial J$ is a Virasoro field of central charge $0$. With respect to this Virasoro element, $J$ is still of conformal weight $1$ but it is no longer primary. The two vectors $G^\pm$ are still primary vectors but their conformal weight changes. It is customary to write $Q := G^+$ and $H=G^-$, then $Q$ has conformal weight $1$ and charge $+1$ and $H$ has conformal weight $2$ and charge $-1$, they are expanded as 
\begin{equation} \label{eq:n=2-newgen}
Q(z) = \sum_{n \in \mathbb{Z}} z^{-1-n} Q_n, \qquad H(z) = \sum_{n \in \mathbb{Z}} z^{-2-n} H_n. 
\end{equation}
These endomorphisms together with \eqref{eq:virasoro2} and \eqref{eq:j-decomp} form a $\mathbb{Z}$-graded Lie superalgebra. Its degree zero part is spanned by $L_0, J_0, Q_0$ and $H_0$ is isomorphic to $\fg\mathfrak{l}(1|1)$. In particular $Q_0^2 = 0$, it commutes with $L_0$, hence it preserves conformal weight and it increases charge by $1$. 
\end{nolabel}
\begin{nolabel}\label{no:cohomology}
Let $V$ be a vertex algebra together with four vectors $T, J, Q$ and $H$ as in \ref{no:topological-twist1}, such that the eigenvalues of $T_0$ and $J_0$ are integer. Then $V$ has two different gradings, one by conformal weight (the eigenvalues of $T_0$) and another one by charge (eigenvalues of $J_0$). We will denote by $V^\bullet$ the graded vector space with respect to the charge grading. The endomorphism $Q_0: V^\bullet \rightarrow V^{\bullet + 1}$ makes $V$ into a complex. The endomorphism $T_0$ preserves degrees and commutes with $Q_0$ hence it is a map of complexes. Finally the endomorphism $H_0: V^\bullet \rightarrow V^{\bullet -1}$ satisfies $[Q_0, H_0] = T_0$, hence it produces an homotopy from the map $T_0$ to zero. It follows that the cohomology $H^\bullet(V, Q_0)$ is concentrated in the sub-complex of conformal weight zero. 
\end{nolabel}
\begin{nolabel}\label{no:n=4}
The $N=4$ vertex algebra is a superconformal vertex algebra that has in addition to the generator $G$ of $N=1$ with central charge $c$, a $\mathfrak{s}\mathfrak{u}_2$ worth of $N=2$, structures, that is we have three even vectors $J_1, J_2, J_3$ each one of primary of conformal weight $1$ with respect to the superconformal structure determined by $G$ and satisfying \eqref{eq:n=2-j-ope}. The remaining OPEs are given by
\begin{equation} \label{eq:n=4-ope} J_i(z) \cdot J_j(w) \sim 2 \varepsilon_{ijk} \frac{J_k(w)}{z-w}, \qquad (\ds J_i)(z) \cdot J_j(w) \sim \varepsilon_{ijk} \frac{(\ds J_k)(w)}{z-w}, \qquad i \neq j\end{equation} 
where $\varepsilon_{ijk}$ is the totally antisymmetric tensor. 

As a usual vertex algebra this algebra has four even generators: the three currents $J_i$ forming an affine $\mathfrak{s}\mathfrak{l}_2$ Kac-Moody vertex algebra, the superpartner of $G$, $L= \tfrac{1}{2} \ds G$ which gives a Virasoro vector. With respect to this vector all the currents $J_i$ are primary of conformal weight $1$. The three super-partners $\ds J_i$ together with $G$ are the four odd generators, they are all primary of conformal weight $3/2$ with respect to $L$ and they form a $\mathbb{C}^2 \oplus \mathbb{C}^{2*}$ representation of $\mathfrak{sl}_2$ with respect to the three currents $J_i$. The OPE between these odd generators can be found in \cite[5.9.7b)]{kac:vertex} for example, we will not need them since they are implied  by applying $\ds$ to the second ope in \eqref{eq:n=4-ope}
\end{nolabel}
\begin{nolabel}\label{no:Spin7}
The Shatashvili-Vafa vertex algebra $SVSpin_7$ associated to $Spin_7$ introduced in \cite{vafa} is the superconformal vertex algebra generated by an $N=1$ vector $G$ of central charge $12$ and an even field $X$ of conformal weight $2$ which is non-primary. The OPEs are (here $L = \tfrac{1}{2} \ds G$ is the Virasoro element)
\begin{equation} \label{eq:svspin7}
\begin{aligned}
L(z) \cdot X(w) &\sim \frac{\partial_w X(w)}{z-w} + \frac{2 X(w)}{(x-w)^2} + \frac{2}{(z-w)^4}, \\ 
G(z) \cdot X(w) &\sim \frac{(\ds X)(w)}{z-w} + \frac{1}{2} \frac{G(w)}{(z-w)^2}, \\ 
X(z) \cdot X(w) &\sim \frac{8 \partial_w X(w)}{(z-w)} + \frac{16 X(w)}{(z-w)^2} + \frac{16}{(z-w)^4} \\
(\ds X)(z) X(w) &\sim \frac{5}{2} \frac{\partial_w (\ds X)(w)}{z-w} + \frac{5}{4} \frac{\partial_w^2 G(w)}{z-w} + 6 \frac{(G\cdot X)(w)}{(z-w)} + \\ 
& \quad 8 \frac{\partial_w X(w)}{(z-w)^2} + \frac{15}{4} \frac{\partial_w G(w)}{(z-w)^2} + \frac{15}{2} \frac{G(w)}{(z-w)^3}
\end{aligned}
\end{equation} 
Notice that these OPEs are non-linear in the generators due to the appearance of the product $(G \cdot X)$ in the last OPE. 

$X$ generates another copy of the Virasoro algebra of central charge $1/2$ inside of $SVSPin_7$. 
\end{nolabel}
\begin{nolabel}\label{no:G2}
The Shatashvili-Vafa vertex algebra $SVG_2$ associated to $G_2$ is the superconformal vertex algebra introduced in \cite{vafa} and generated by an $N=1$ vector $G$ of central charge $c = 21/2$, an odd vector $\Phi$ of conformal weight $3/2$ and an even vector $X$ of conformal weight $2$. Their superpartners are $L = \tfrac{1}{2} \ds G$ which is a Virasoro vector (hence conformal weight $2$), $K := \ds \Phi$ of conformal weight $2$ and $M:=\ds X$ of conformal weight $5/2$. $\Phi$ is primary with respect to the full $N=1$ structure but $X$ is not. 
The remaining OPEs are (all other ones can be deduced by supersymmetry applying $\ds$ to these ones)
\begin{equation}\label{eq:svg2}
\begin{aligned}
\Phi(z) \cdot \Phi(w) &\sim \frac{6 X(w)}{z-w} - \frac{7}{(z-w)^3}, \\
\Phi(z) \cdot K(w) &\sim -\frac{3}{2} \frac{2M(w) + \partial_w G(w)}{z-w} - \frac{3 G(w)}{(z-w)^2} \\
G(z) \cdot X(w) &\sim \frac{M(w)}{z-w} - \frac{1}{2} \frac{G(w)}{(z-w)^2}, \\ 
L(z) \cdot X(w) &\sim \frac{\partial_w X(w)}{z-w} + \frac{2 X(w)}{(z-w)^2} - \frac{7}{4} \frac{1}{(z-w)^4}, \\ 
\Phi(z) \cdot X(w) &\sim - \frac{5}{2} \frac{\partial_w \Phi(w)}{z-w} - \frac{15}{2} \frac{\Phi(w)}{(z-w)^2}, \\
\Phi(z) \cdot M(w) &\sim \frac{5}{2} \frac{ \partial_w K(w) - 6(G\cdot \Phi)(w)}{z-w} - \frac{9}{2} \frac{K(w)}{(z-w)^2}, \\
X(z) \cdot X(w) &\sim -5\frac{\partial_w X(w)}{z-w} - 10 \frac{X(w)}{(z-w)^2} + \frac{35}{4} \frac{c}{(z-w)^4}, \\ 
X(z) \cdot M(w) &\sim \frac{4 (G \cdot X)(x)}{z-w} - \frac{7}{2} \frac{\partial_w M(w)}{z-w} - \frac{3}{4} \frac{\partial^2_w G(w)}{z-w} - \\ 
& \quad \frac{5M(w)}{(z-w)^2} - \frac{9}{4} \frac{\partial_w G(w)}{(z-w)^2} - \frac{9}{2} \frac{G(w)}{(z-w)^3}. 
\end{aligned}
\end{equation}
This is another example of a non-linearly generated supersymmetric vertex algebra. 

The pair $\Phi$, $X$ generates another $N=1$ structure of central charge $7/10$ which does not preserve (commute) with the $N=1$ structure generated by $G, L$. 
\end{nolabel}
\begin{rem} In all the examples of superconformal algebras that we have listed above, that is $N=2$, $N=4$, $Spin_7$ and $G_2$, the corresponding algebras are determined by the lowest conformal weight vectors and their superpartner if we where to allow positive products in \eqref{eq:ordered1}. That is, all the other generators are obtained from the OPE of these lowest conformal weight generators. Indeed, in the case of $N=2$ as in \ref{no:n=2} the generator $G$ can be defined by the RHS of \eqref{eq:n=2-j-ope}. The same is true for $N=4$ where knowing simply $J_i$ (actually only two of them will suffice) we obtain $G$ from the same equation \eqref{eq:n=2-j-ope}. The case of $SVSpin_7$ is generated simply by $X$ as we can read $G$ as the single pole term in the ope of $X$ with its superpartner, or the last equation in \eqref{eq:svspin7}. Similarly, in the case of $SVG_2$ we obtain all other fields from $\Phi$ and its superpartner, as $X$ and $G$ can be read of from the first and second equations in \eqref{eq:svg2}. 
\label{rem:generating-currents}
\end{rem}
\section{Chiral de Rham Complex} 
\begin{nolabel}\label{no:beta-gamma}
The $n$-dimensional $bc$-$\beta\gamma$ system is the supersymmetric vertex algebra generated by $n$ even fields $\{\gamma^i\}_{i=1}^n$, $n$ odd fields $\{b_i\}_{i=1}^n$ and their superpartners:
\[ \ds \gamma^i = c^i, \qquad \ds b_i = \beta_i, \] with the only non-vanishing OPEs:
\[ \beta_i(z) \cdot \gamma^j(w) \sim \frac{\delta_i^j}{z-w}, \qquad b_i(z) \cdot c^j(w) \sim \frac{\delta_i^j}{z-w}. \] 
The odd vector 
\begin{equation} \label{eq:g-global-sec} G = \sum_{i=1}^{n} \left( \beta_i c^i + b_i \partial \gamma^i \right). \end{equation}
defines a superconformal structure of central charge $c=3n$.  With respect to this structure, the generators $\gamma^i$ are primary of conformal weight $0$ while the generators $b_i$ are primary of conformal weight $1/2$. It follows that their superpartners $c^i$ and $\beta_i$ have respectively conformal weight $1/2$ and $1$. 

The subalgebra generated by $\gamma^i$ and their superpartners is (super) commutative and is isomorphic to $\mathbb{C}[\gamma^i, \ds \gamma^i, \ds^2 \gamma^i, \cdots ]$. In particular the commutative algebra $\mathbb{C}[\gamma^1,\cdots, \gamma^n]$ is a subalgebra. One can extend the vertex algebra structure on the $n$-dimensional $bc$-$\beta\gamma$ system by replacing this commutative algebra to allow for arbitrary smooth (resp. holomorphic) functions on $\gamma^i$, imposing that $\ds$ is an odd derivation. In what follows by the smooth (resp. holomorphic) $n$-dimensional $bc$-$\beta\gamma$ system we will mean the corresponding vertex algebra. 
\end{nolabel}
\begin{nolabel}\label{no:coords1}
Let $M$ be a smooth $n$-dimensional manifold and $\mathbb{R}^n \simeq U \subset M$ be a coordinate patch with coordinates $\left\{ x^i \right\}_{i=1}^n$ and let $U'$ be another coordinate patch with coordinates $\{\tilde{x}^i\}_{i=1}^n$. On the intersection $U \cap U'$ we have the corresponding change of coordinates and their inverses: 
\begin{equation} \label{eq:changes-def} \tilde{x}^i = f^i(x^1, \cdots, x^n), \qquad x^i = g^i(\tilde{x}^1, \cdots, \tilde{x}^n). \end{equation}
In \cite{malikov} the authors introduced a sheaf of vertex algebras $\Omega^{ch}_M$ that locally is given by the smooth $n$-dimensional $bc$-$\beta\gamma$ system. That is, locally $\Omega^{ch}(U)$ is generated by the fields $\left\{ \gamma^i, b_i \right\}_{i=1}^n$ and their superpartners and $\Omega^{ch}(U')$ by $\left\{ \tilde{\gamma}^i, \tilde{b}_i \right\}_{i=1}^n$ and their super-partners. On the intersection $U \cap U'$ we have the relations:
\begin{equation} \label{eq:change1} \tilde{\gamma}^i = f^i(\gamma^1, \cdots, \gamma^n), \qquad \tilde{b}_i = \sum_{j = 1}^n \left( \frac{\partial g^j}{\partial \tilde{x}^i}\left(\tilde{\gamma}^1, \cdots, \tilde{\gamma}^n \right) \right)b_j, \qquad \tilde{\ds} = \ds. 
\end{equation}
The transformation properties for $\left\{ c^i, \beta_i \right\}$ are obtained by applying $\ds$ to these expressions and using that $\ds$ is a derivation, that is 
\begin{equation} \label{eq:change2} \tilde{c}^i = \ds \tilde{\gamma}^i = \ds f^i(\gamma^1, \cdots, \gamma^n) = \sum_{j=1}^n \left( \frac{\partial f^i}{\partial x^j} (\gamma^1, \cdots, \gamma^n) \right) \ds \gamma^j = 
\sum_{j=1}^n \left( \frac{\partial f^i}{\partial x^j} (\gamma^1, \cdots, \gamma^n) \right) c^j,
\end{equation}
and 
\begin{multline}\label{eq:change3}
\tilde{\beta}_i = \ds \tilde{b}_i = \ds \left (\sum_{j = 1}^n \left( \frac{\partial g^j}{\partial \tilde{x}^i}\left(\tilde{\gamma}^1, \cdots, \tilde{\gamma}^n \right) \right)b_j\right) = \\ \sum_{k,j=1}^n \left( \frac{\partial^2 g}{\partial \tilde{x}^j \partial \tilde{x}^k}(\tilde{\gamma}^1, \cdots, \tilde{\gamma}^n) \ds \tilde{\gamma}^k \right) b_j + \sum_{j = 1}^n \left( \frac{\partial g^j}{\partial \tilde{x}^i}\left(\tilde{\gamma}^1, \cdots, \tilde{\gamma}^n \right) \right) \ds b_j = \\ 
= \sum_{j,k,l=1}^n \left( \frac{\partial^2 g}{\partial\tilde{x}^j \partial\tilde{x}^k}(\tilde{\gamma}^1, \cdots, \tilde{\gamma}^n) \frac{\partial f^k}{\partial x^l}(\gamma^1,\cdots, \gamma^n) c^l \right) b_j + \sum_{j=1}^n \left( \frac{\partial g^j}{\partial{\tilde{x}^i}}(\tilde{\gamma}^1, \cdots, \tilde{\gamma}^n) \right) \beta_j
\end{multline}
It was proved in \cite{malikov} that the changes of coordinates \eqref{eq:change1}--\eqref{eq:change3} are given by automorphisms of the vertex algebra structure, that is, the only nontrivial OPEs are given by 
\[ \tilde{\beta}_i(z) \cdot \tilde{\gamma}^j(w) \sim \frac{\delta_i^j}{z-w}, \qquad \tilde{b}_i(z) \cdot \tilde{c}^j(w) \sim \frac{\delta_i^j}{z-w}. \] 
\end{nolabel}
\begin{nolabel}\label{no:courant1}
Here is a coordinate-free description of the chiral de Rham complex. The following was proved in \cite{heluani8} and is essentially a remark based on the work of Bressler \cite{bressler} who found the relation between the \emph{vertex algebroids} of \cite{gms} and Courant algebroids (see also \cite{beilinsondrinfeld} where the concept of chiral differential operators and their corresponding envelopes where studied and \cite{roytenberg-dorfman} where the classical limits where studied in relation to Poisson vertex algebras). Let $E$ be a Courant algebroid on $M$. Then there exists a unique sheaf of supersymmetric vertex algebras $\Omega=\Omega(E)$ on $M$ together with maps 
\[ \iota: C^\infty(U) \hookrightarrow \Omega(U) \hookleftarrow \Gamma(U, E): j\]
where the image of $\iota$ consists of even vectors and that of $j$ of odd vectors such that 
\begin{enumerate}
\item $\iota$ is compatible with the algebra structure of $C^\infty(U)$:  $\iota(1) = \vac \in \Omega(U)$ and $\iota (fg) = \iota(f) \iota(g)$ for all $f,g \in C^\infty(U)$. 
\item $j$ and $\iota$ are compatible with the $C^\infty(U)$-module structure on $\Gamma(U, E)$:  $j (f e) = i(f) \cdot j(a)$, for all $f \in C^\infty(U), a \in \Gamma(U, E)$
\item There is a compatibility between the supersymmetry generator in $\Omega(E)$ and the differential $d : C^\infty (U) \rightarrow \Gamma(U, E)$ which is part of the Courant algebroid datum: $\ds \iota (f) = j(df)$.
\item There is a compatibility between the Dorfman bracket and symmetric pairing on $E$ and the OPE on $\Omega(E)$:
\[ (j(a))(z) \cdot (j(b))(w) \sim \frac{(\iota (\langle a,b\rangle)) (w)}{z-w}, \qquad 
(\ds j(a))(z) \cdot (j(b))(w) \sim \frac{(j ([ a,b])) (w)}{z-w}.\]
\end{enumerate}
The sheaf $\Omega(E)$ is universal with the above properties, in the sense that if there is another triple $(\Omega'(E), \iota', j')$ with the same properties, then there exists a unique morphism of sheaves of vertex algebras $\Omega(E) \rightarrow \Omega'(E)$ intertwining $\iota, j$ with $\iota', j'$. 

When $E$ is the standard Courant algebroid $T_M \oplus T^*_M$ we have an identification of $\Omega(E)$ and $\Omega^{ch}_M$. Locally, on the coordinate chart $U$ with coordinates $\left\{ x^i \right\}_{i=1}^n$ the identification is given by 
\[ \iota \left( f(x^1,\cdots,x^n) \right) \mapsto f(\gamma^1, \cdots, \gamma^n), \qquad j \left( \frac{\partial}{\partial x^i} \right) = b_i. \]
Notice that by c) above we have $j (dx^i) \mapsto c^i$. This is the \emph{naive-embedding} of \ref{no:sheaf}. 
\end{nolabel}

\begin{nolabel}\label{no:affine}
The construction in \ref{no:courant1} is interesting even when $M$ reduces to a point. In this case a Courant algebroid $E$ is given by a finite dimensional Lie algebra $\fg$ together with a symmetric invariant and non-degenerate bilinear form $\langle, \rangle$. The corresponding vertex algebra $\Omega(E)$ is simply given by $V(\fg_{super})$ as in \ref{no:vgsuper}. 
\end{nolabel}
\begin{nolabel}\label{no:holomorphic}
The same construction as in \ref{no:courant1} can be realized when $E$ is a holomorphic  Courant algebroid over a holomorphic manifold or a complex algebraic Courant algebroid over a smooth algebraic complex variety, by replacing $C^\infty(U)$ by holomorphic (resp. complex algebraic) functions in $U$ and $\Gamma(U,E)$ by holomorphic (resp. complex algebraic) sections of $E$ over $U$. In the case of the standard Courant algebroid $TM \oplus T^*M$ we would obtain the holomorphic or algebraic versions of $\Omega^{ch}_M$. 
\end{nolabel}
\begin{nolabel}
By construction there is an obvious embedding $T^*M \rightarrow \Omega^{ch}_M$ which is given either by c) in \ref{no:courant1} or by the local assignment \[ \sum_{i=1}^n f_i(x^1,\cdots,x^n) dx^i \mapsto \sum f_i(\gamma^1, \cdots, \gamma^n) c^i. \]  

One can extend this to an embedding 
\begin{equation} \label{eq:embedding1} 
\wedge^\bullet T^*M \rightarrow \Omega^{ch}_M. 
\end{equation} 
It was proved in \cite{malikov} that one can endow $\Omega^{ch}_M$ with a $\mathbb{Z}$-grading and an odd endomorphism $Q_0: \Omega^{ch,\cdot}_M \rightarrow \Omega^{ch,\cdot + 1}_M$ such that $Q_0^2 = 0$ and making the embedding \eqref{eq:embedding1} a quasi-isomorphism. 

If $M$ is orientable then the local sections defined in a coordinate patch $(U, \left\{ x^i \right\}_{i=1}^n)$ by 
\begin{equation} \label{eq:current} J = \sum_{i=1}^n c^i b_i \in \Omega^{ch}(U),\end{equation}
are shown in \cite{malikov} to glue to a global section $J \in \Omega^{ch}(M)$. The corresponding zero mode $J_0 \in \End \Omega^{ch}(M)$ is semisimple with integer eigenvalues. This defines a $\mathbb{Z}$-grading on $\Omega^{ch}(M)$. In fact, even when $M$ is not orientable, in which case $J$ is not globally well defined, its zero mode $J_0$ still is. 

A simple way of seeing this grading is as follows: One assigns the even vectors $\gamma^i$ and $\beta_i$ degree $0$ and the odd generators $c^i$ (resp. $b_i$) degree $+1$ (resp. $-1$) and extend this to $\Omega^{ch}_M$ by declaring the normally ordered product $\cdot$ and the translation operator $\partial$ to be of degree $0$. One immediately see that the transformation formulas \eqref{eq:change1}--\eqref{eq:change3} are compatible with these degree assignments and we obtain this way a $\mathbb{Z}$-grading on the vertex algebra. 

Notice however that this $\mathbb{Z}$-grading is not compatible with the supersymmetric structure $\ds$. Indeed locally we have $\gamma^i$ and $\partial \gamma^i$ are of degree $0$ but $\partial \gamma^i = \ds^2 \gamma^i = \ds c^i$ and $c^i$ is of degree $1$. 

The differential $Q_0$ is constructed by defining on generators $Q_0 \gamma^i = c^i$ and $Q_0 b_i = \beta_i$ and declaring it to be a derivation such that $Q_0^2 = 0$. We see that $Q_0$ agrees with $\ds$ on the generators $\gamma^i$ and $b_i$ but not on $c^i$ and $\beta_i$, hence $Q_0$ is a derivation of the vertex algebra structure of $\Omega^{ch}_M$ but does not preserve its supersymmetric structure. 
\end{nolabel}
\begin{nolabel}\label{no:topological-twist2}
The reason why $Q_0$ does not preserve the supersymmetric structure of $\Omega^{ch}_M$ involves a topological twist as in \ref{no:topological-twist1}. In addition to the field $J$ of \eqref{eq:current} that exists when $M$ is orientable, there is another global field that one can define locally by \eqref{eq:g-global-sec}. 

It is straightforward to check that the OPEs of these fields $J$ and $G$ and their superpartners satisfy the OPEs as in \ref{no:n=2} with central charge $c=3n$ hence they induce an embedding of the $N=2$ superconformal algebra of central charge $3n$ into $\Omega^{ch}_M$. As explained in \ref{no:n=1} the zero mode of $G$ coincides with $\ds$. 

As in \ref{no:topological-twist1} one may consider the same $N=2$ vertex algebra but with a different choice of Virasoro 
\begin{equation} \label{eq:new-virasoro}
T = L + \tfrac{1}{2} \partial J, 
\end{equation} 
of central charge $0$, in which case the two vectors $Q$ and $H$ of \eqref{eq:n=2-newgen} acquire different conformal weights $1$ and $2$ respectively. The zero mode $Q_0$ of the field corresponding to $Q$ is an endomorphism satisfying $Q_0^2 = 0$. In terms of the local generators we have
\begin{equation} \label{eq:q-and-h} Q = \sum_{i=1}^n \beta_i c^i, \qquad H = \sum_{i=1}^{n} b_i \partial \gamma^i, \end{equation}
and we have $\ds = Q_0 + H_{-1}$. 

As in \ref{no:topological-twist1} the cohomology $H^\bullet(\Omega^{ch}_M, Q_0)$ is concentrated in conformal weight $0$. With respect to $T$ given by \eqref{eq:new-virasoro}, the fields $\{ \gamma^i, c^i \}_{i=1}^n$ have conformal weight $0$ while the other generators $\left\{ \beta_i, b_i \right\}_{i=1}^n$ have conformal weight $1$. It follows that locally the sub-complex of conformal weight zero in $\Omega^{ch}_M$ is simply the (super)-commutative algebra generated by $\left\{ \gamma^i, c^i \right\}$. This is identified with $\wedge^\bullet T^*_M$ via \eqref{eq:embedding1}. Under this identification $Q_0$ is identified with the de Rham differential. Indeed $c^i$ (or rather its zero mode) simply acts by multiplication by $c^i$ which in turn is identified with $dx^i$, while $\beta_i$ (or rather its zero mode) acts on forms as $\partial/\partial x^i$. 

When the manifold $M$ is not orientable, the section $G$ is not defined globally, however the zero modes $Q_0$, $H_0$ are. This is the way is the way in which it was proved in \cite{malikov} that the cohomology $H^\bullet(\Omega^{ch}_M, Q_0)$ equals de de Rham cohomology of $M$. 
\end{nolabel}
\begin{nolabel}\label{no:holomorphic-orientable}
In the holomorphic or complex analytic setting, the existence of the section \eqref{eq:g-global-sec} is guaranteed when $M$ is holomorphically orientable, that is, when there exists a global holomorphic volume form, or when the first Chern class of the holomorphic tangent bundle of $M$ vanishes. If in addition $M$ is simply connected, this is a Calabi-Yau manifold. Thus, if $M$ is a Calabi-Yau manifold, the holomorphic sections of the holomorphic chiral de Rham complex $\Omega^{ch}_M$ form a vertex algebra with an $N=2$ superconformal structure of central charge $3 \dim_\mathbb{C} M$. This was the setting in the original work \cite{malikov}. 

If in addition $M$ is compact, the space of global holomorphic sections $H^0(M, \Omega^{ch}_M)$ is a graded vertex algebra (both by conformal weight and by charge) with finite dimensional conformal weight eigenspaces.  More generally, the full sheaf cohomology $V:=H^*(M, \Omega^{ch}_M)$ is a superconformal vertex algebra with central charge $c=3 \dim_{\mathbb{C}} M$. Letting $T_0$ be the zero mode of the Virasoro element $T$ of \eqref{eq:new-virasoro}, and $J_0$ the zero mode of \eqref{eq:current}, Borisov and Libgober showed in \cite{borisov-libgober} that
\begin{equation} \label{eq:borisov-libgober} \mathcal{Ell}_M(\tau,\alpha) = y^{-c/6} tr_{V} q^{T_0} y^{J_0}, \qquad q=e^{2  \pi i \tau}, y = e^{2 \pi i \alpha}, \qquad \tau \in \mathbb{H}, \alpha \in \mathbb{C}\end{equation} 
converges to a weak Jacobi Form of weight $0$ and index $\frac{1}{2} \dim_{\mathbb{C}} M$ which is called the \emph{two variable elliptic genus of $M$}. This is a cobordism invariant of $M$. 
\end{nolabel}
\begin{nolabel}\label{no:two-sectors}
In the $C^\infty$ context however, the algebra of smooth sections $\Gamma(U, \Omega^{ch}_M)$ is too large. When $M$ is orientable it has an $N=2$ superconformal structure of central charge $3 \dim_{\mathbb{R}} M$, but even if $M$ is compact, the space of conformal weight $0$ elements with respect to the untwisted Virasoro element $L$ (respectively with respect to $T$ given by \ref{eq:new-virasoro}) consists of all smooth functions on $M$ (resp. smooth differential forms on $M$). There are two approaches in this situation, either one focuses on holomorphic/antiholomorphic sections or one uses a different embedding instead of \eqref{eq:embedding1}. Let us describe the first approach postponing the description of the second approach after a general discussion on Boson-Fermion systems in \ref{no:boson-fermion-vs-ghosts}. 

If $M$ is a holomorphic $2n$-manifold, the cotangent bundle naturally splits as $T^* = T^*_{1,0} \oplus T^*_{0,1}$. The embedding \eqref{eq:embedding1} is compatible with this decomposition, in the sense that locally, choosing holomorphic (resp. anti-holomorphic) coordinates $\left\{ z^\alpha \right\}_{\alpha = 1}^n$ (resp. $\left\{ z^{\bar \alpha} \right\}$ ) we have locally $4n + 4n$ generators of $\Omega^{ch}_M$ given by $\left\{ \gamma^\alpha, c^\alpha, \beta_\alpha, b_\alpha \right\}_{\alpha=1}^n$ and their anti-holomorphic counterparts. Since we can cover $M$ by such coordinate patches and on intersection we have biholomoprhic changes of coordinates \eqref{eq:changes-def} and the transformation properties \eqref{eq:change1}--\eqref{eq:change3} are compatible with these changes. In short, the sheaf $\Omega^{ch}_M$ is naturally the tensor product of two commuting copies of the holomorphc chiral de Rham complex of $M$. Naming this last one by $\Omega^{ch, hol}_M$ to differentiate it from its $C^\infty$-version, and $\overline{\Omega}^{ch,hol}_M$ the complex generated by the anti-holomorphic counterparts, we have 

\begin{equation} \label{eq:tensor-iso} \Omega^{ch}_M \simeq \Omega_M^{ch,hol} \otimes \overline{\Omega}^{ch, hol}_M. \end{equation} 

There is a subtlety at the level of functions in that one cannot express smooth functions as tensor products of holomorphic times antiholomorphic functions. We need not be concerned about this issue that can be resolved either by completing the tensor product or tensoring over the commutative vertex algebra generated by smooth functions of $M$. 

In this situation, \eqref{eq:current} can be written as 
\begin{equation} \label{eq:current2}
J = \sum_\alpha c^{\alpha} b_\alpha + \sum_{\bar \alpha} c^{\bar\alpha} b_{\bar\alpha} = J^{hol} + \bar{J}^{hol}, 
\end{equation} 
where each summand is well defined and commutes with the other one. 
Similarly, \eqref{eq:g-global-sec} can be written as
\begin{equation} \label{eq:g-global-sec2} 
G = \left( \sum_{\alpha = 1}^n \left( \beta_\alpha c^{\alpha} + b_\alpha \partial\gamma^{\alpha} \right)\right) + \left( \sum_{\bar\alpha =1}^{n} \left( \beta_{\bar\alpha}c^{\bar\alpha} + b_{\bar\alpha} \partial \gamma^{\bar{\alpha}} \right) \right) = G^{hol} + \bar{G}^{hol}, 
\end{equation}
however, each of the two summands is not well defined globally. If $M$ is holomorphically orientable, one can easily show (as it was done in \cite{malikov}) that each summand is a well defined section of $\Omega^{ch}_M$

These four sections and their corresponding superpartners generate two commuting copies of the $N=2$ superconformal algebra of central charge $3 \dim_{\mathbb{C}}M$ which not surprisingly under the isomorphism \eqref{eq:tensor-iso} simply correspond to the holomorphic $N=2$ structure of \ref{no:holomorphic-orientable}, one in each factor. 
\end{nolabel}
\begin{nolabel}\label{no:drawback1}
Since we have two commuting copies of $N=2$ we can make a topological twist in each sector and compute the cohomology with respect to the corresponding differential. Let $\overline{T}, \overline{J}, \overline{Q}$ and $\overline{H}$ be the generators of the $N=2$ inside of the $\overline{\Omega}^{ch,hol}_M$ factor.  If we consider the differential $\overline{Q}_0$ given as the zero mode of (cf. \eqref{eq:q-and-h})
\[ \overline{Q} = \sum_{\bar\alpha} \beta_{\bar\alpha} c^{\bar\alpha}, \] 
We know that the cohomology $H(\Omega^{ch}_M, \overline{Q}_0)$ is a sheaf of vertex algebras that is concentrated in conformal weight zero for $\overline{T}_0$. The space of conformal weight zero vectors for this Virasoro element is generated locally by all the local sections $c^\alpha, \partial \gamma^\alpha, b_\alpha, \beta_\alpha$ of $\Omega^{ch,hol}$ as well as all smooth sections of 
\begin{equation} \label{eq:dolbeaut1} \wedge^\bullet T^*_{0,1}. \end{equation}
Indeed notice that $\overline{T}$ commutes with all of $\Omega_M^{ch,hol}$, it is convenient to consider the smooth functions separatedly and that is why we include them in \eqref{eq:dolbeaut1}
The action of $\overline{Q}_0$ on $\Omega^{ch,hol}_M$ is zero, while on \eqref{eq:dolbeaut1} it coincides with the Dolbeaut differential. Since locally we have a $\bar{\partial}$ Poincaré lemma, the cohomology of \eqref{eq:dolbeaut1} is the sheaf $\cO_M$ of holomorphic functions of $M$, therefore we obtain 
\begin{equation} \label{eq:cohom1}
H(\Omega^{ch}_M, \overline{Q}_0) \simeq \Omega^{ch, hol}_M, 
\end{equation}
In other words, the obvious or \emph{naive} embedding of sheaves of vertex algebras:
\[ \Omega^{ch, hol} \hookrightarrow (\Omega^{ch}_M, \overline{Q}_0), \]
is a quasi-isomorphism. Since the remaining $N=2$ commutes with $\overline{Q}_0$ we see that it survives cohomology and it coincides with the $N=2$ structure of \ref{no:holomorphic-orientable}.  We can now look at the sheaf cohomology on $M$, in other words:
\begin{equation} \label{eq:ell2}  R\Gamma^\bullet(M, H(\Omega^{ch}_M, \overline{Q}_0)) \simeq R\Gamma^\bullet(M, \Omega^{ch,hol}_M), \end{equation}
providing a vertex algebra with finite dimensional graded pieces, and hence we can look at the modular character as in \eqref{eq:borisov-libgober}. 

Note however that we took the cohomology of the \emph{sheaf} $\Omega^{ch}_M$ with respect to the sheaf endomorphism $\overline{Q}_0$. We could perform this procedure in the reverse order,  taking first global sections (the higher sheaf cohomologies vanish in the smooth situation) and then the  cohomology with respect to $\overline{Q}_0$. A standard argument with spectral sequences relates these two cohomologies. We obtain that the elliptic genus of $M$ can be computed as 
\begin{equation}\label{eq:ell3}
\mathcal{Ell}_M(\tau,\alpha) = y^{- \frac{3}{4} \dim_{\mathbb{R}}M} \tr_{H \left( C^\infty(M, \Omega_M^{ch}), Q_0^- \right)} q^{T_0^+} y^{J^+_0}, \qquad q=e^{2\pi i \tau},\: y = e^{2 \pi i \alpha}.
\end{equation}

\end{nolabel}
\begin{nolabel}\label{no:drawback2}
One drawback of the approach described in \ref{no:drawback1} is that it requires us to deal with technical subtleties as completing the tensor product of holomorphic/anti-holomorphic funtions, as well as it requires us to have special coordinate systems (holomorphic in this case). What can we do on other special holonomy cases? Can we obtain a vertex operator algebra with finite dimensional energy spaces? can we attach a modular form/function to such a manifold by taking traces of these vertex algebras? These questions have been studied intensely in the physics literature in the past. We mention specifically the approach of de Boer, Naqvi and Shommer \cite{boer} that is close to our needs in the case of $G_2$-manifolds. 

Another approach to seeing the two commuting $N=2$ structures on the chiral de Rham complex of Calabi-Yau manifolds started in \cite{heluani8} and extended for for all special holonomy manifolds in \cite{heluaniholonomy} consists on using a Riemannian metric on $M$ instead to produce two different and non-commuting embeddings replacing \eqref{eq:embedding1}. Before we describe this we digress on the Boson-Fermion system.
\end{nolabel}
\begin{nolabel}\label{no:boson-fermion-vs-ghosts}
The $bc$-$\beta-\gamma$ system described in \ref{no:beta-gamma} is sometimes called a \emph{ghost} system in the physics literature. This kind of vertex algebras appears in the \emph{first order} formalism, or when trying to use a Hamiltonian approach to quantize the sigma model \cite{heluani-sigma}. In the usual approach however the basic vertex algebra of free fields is that of the $n$-dimensional \emph{Boson-Fermion} system. 

Let $V$ be a vector space with a symmetric non-degenerate bilinear form $\langle, \rangle$. The Boson-Fermion system based on $V$ consists of the supersymmetric vertex algebra $BF(V)$ generated by odd vectors $\bar{v}$, $v \in V$ satisfying the OPEs
\begin{equation} \label{eq:boson-fermion} \bar{v}(z) \cdot \bar{u}(w) = \frac{\langle v,u \rangle}{z-w},  \qquad (\ds \bar{v})(z) \cdot \bar{u}(w) \sim 0. \end{equation}
If we denote the superpartners $v:=\ds \bar{v}$ it follows by applying $\ds$ to \eqref{eq:boson-fermion}
\begin{equation} \label{eq:boson-fermion2} v(z) \cdot u(w) \sim \frac{\langle v, u\rangle}{(z-w)^2}. \end{equation} 

Let now $\left\{ v_i \right\}_{i=1}^n$ be a basis for $V$ and let $\left\{ v^i \right\}_{i=1}^n$ be the dual basis with respect to $\langle,\rangle$. Then the vector
\begin{equation}\label{eq:boson-fermion-g}
G = \sum_{i=1}^n \bar{v}_i \cdot v^i,
\end{equation} 
defines an $N=1$ superconformal structure of central charge $c=3n/2$ in $BF(V)$. With respect to this structure the generators $\bar{v}$ are primary of conformal weight $1/2$. 

Let now $g_{ij}$ be the matrix of $\langle,\rangle$ with respect to this basis, so that \eqref{eq:boson-fermion}--\eqref{eq:boson-fermion2} read now:
\begin{equation} \label{eq:boson-fermion3} \bar{v}_i(z) \cdot \bar{v}_j(w) \sim \frac{g_{ij}}{z-w}, \qquad \bar{v}_i(z) \cdot v_j(w) \sim 0, \qquad v_i(z) \cdot v_j(w) \sim \frac{g_{ij}}{(z-w)^2}. \end{equation}

We have two different commuting embeddings of $BF(V)$ into the $bc-\beta\gamma$ system given as the graph of $g$. 
\begin{equation} \label{eq:embedding2} \bar{v}_i \mapsto \frac{b_i \pm \sum_{j=1}^n g_{ij} c^j}{\sqrt{\pm 2}}. \end{equation}
Notice that this forces their superpartners to be 
\begin{equation}\label{eq:embedding3}
v_i \mapsto \frac{\beta_i \pm \sum_{j=1}^{n} g_{ij} \partial \gamma^i}{\sqrt{\pm 2}}. 
\end{equation}
We can write these embeddings as a single embedding from a tensor product $BF(V) \otimes BF(V)$ into the $bc-\beta\gamma$ system. Notice that the image of this embedding consists of the vacuum vector and all the vectors in the $bc-\beta\gamma$ system of conformal weight higher than $0$ with respect to the conformal structure determined by \eqref{eq:g-global-sec}. The only fields that are missing in the image are the vectors $\gamma^i$. In other words, the subalgebra of the $bc-\beta\gamma$ system generated by the odd vectors $c^i, b_i$ and their superpartners $\partial \gamma^i, \beta_i$ consists of two copies of the Boson-Fermion system on an $n$-dimensional vector space. 

Notice that the image of the superconformal structure \eqref{eq:boson-fermion-g} under these two embeddings produces two commuting $N=1$ conformal structures $G_\pm$ in the $bc-\beta\gamma$ system. Their sum coincides with \eqref{eq:g-global-sec}. Indeed we have 
\[ \bar{v}^i \mapsto \frac{\sum_{j=1}^{n} g^{ij}b_j \pm c^i}{\sqrt{\pm 2}}, \]
where we use the inverse metric $g^{ij}$ so that \[ \sum_{j=1}^n g^{ij}g_{kj} = \delta^i_k. \]
It follows that the two images of \eqref{eq:boson-fermion-g} are given by
\begin{multline}\label{eq:boson-fermion-gpm}
G_\pm = \pm \frac{1}{2} \sum_{i=1}^n \left( b_i \pm \sum_{j=1}^n g_{ij} c^j \right) \left( \sum_{j=1}^n g^{ij} \beta_j \pm \partial \gamma^i \right) = \\  \frac{1}{2} \left( b_i \partial \gamma^i + \beta_i c^i \right) \pm \frac{1}{2} \sum_{i,j = 1}^n\left( g^{ij} b_i \beta_j + g_{ij} c^i \partial \gamma^j  \right), 
\end{multline} 
such that $G = G_+ + G_-$. Notice however that these two commuting $N=1$ structures are not the same as \eqref{eq:g-global-sec2} obtained from the decomposition \eqref{eq:tensor-iso}. 
\end{nolabel}
\begin{nolabel}\label{no:difficulty-in-cdr}
Let $M$ be an $n$-dimensional smooth manifold and $g$ a non-degenerate smooth metric (of any signature) on $TM$. At each point $x \in M$ we have the vector space $V = T_xM$ with its bilinear non-degenerate pairing $\langle, \rangle = g_x$. We can try to repeat the construction of \ref{no:boson-fermion-vs-ghosts} to relate the free Boson-Fermion system with the $bc-\beta\gamma$ system and therefore with $\Omega^{ch}_M$. 

There are immediate problems arising from the fact that the metric is not constant. 
Indeed in the local coordinate patch $(U, \left\{ x^i \right\}_{i=1}^n)$ we have $g_{ij} = g_{ij}(x^1, \cdots,x^n)$. Inside of $\Omega^{ch}_M(U)$ we have the corresponding local section $g_{ij}(\gamma^1,\cdots,\gamma^n)$. The local sections given by the RHS of \eqref{eq:embedding2} still satisfy the OPE given by the first equation of \eqref{eq:boson-fermion3}. However, their superpartners are not given by the RHS of \eqref{eq:embedding3} but instead 
\begin{equation}\label{eq:super-partners}
v_i \mapsto \ds \left( \frac{b_i \pm \sum_{j=1}^n g_{ij} c^j}{\sqrt{\pm 2}} \right) = \left(  \frac{\beta_i \pm \sum_{j=1}^{n} g_{ij} \partial \gamma^j}{\sqrt{\pm 2}} \right) \pm \frac{\sum_{j,k = 1}^n \frac{\partial g_{ij}}{\partial x^k}(\gamma^1, \cdots,\gamma^n) c^k c^j }{\sqrt{\pm 2}}, 
\end{equation}
where in the second term we have used that $\ds$ is a derivation and we notice that we do not need to worry about the parenthesis in the product. 

Here is the first complication in trying to pass from a flat space with a constant metric $g_{ij}$ to the chiral de Rham complex of a manifold with a non-constant metric $g_{ij}(x^1,\cdots,x^n)$: the appearance of terms like the second term in \eqref{eq:super-partners} makes it very hard to write down local expressions for sections that will glue into global sections of $\Omega^{ch}_M$. 

If one introduces the Levi-Civita connection  $\nabla$ such that the metric is parallel: $\nabla g = 0$ we can trade the derivatives of the metric in the second term in \eqref{eq:super-partners} for terms that are linear in $g_{ij}$ but that are multiplied by the Christoffel symbols $\Gamma_{ij}^k$ of $\nabla$. 

Note however that in order to check the second equation in \eqref{eq:boson-fermion3} we need to compute the OPE:
\begin{multline}
\left( \frac{b_i \pm \sum_{j=1}^n g_{ij} c^j}{\sqrt{\pm 2}} \right)(z) \cdot \left[ \left(   \frac{\beta_j \pm \sum_{j=1}^{n} g_{jk} \partial \gamma^k}{\sqrt{\pm 2}} \right) \pm \frac{\sum_{k,l = 1}^n \frac{\partial g_{jk}}{\partial x^l}(\gamma^1, \cdots,\gamma^n) c^l c^k }{\sqrt{\pm 2}} \right](w) \sim  \\ \frac{1}{2 (z-w)} \left( \sum_{k=1}^n \frac{\partial g_{jk}}{\partial x^i}(\gamma^1,\cdots,\gamma^n) c^k - \sum_{l=1}^n \frac{\partial g_{ji}}{\partial x^l}(\gamma^1,\cdots,\gamma^n) c^l  - \sum_{k=1}^n \frac{\partial g_{ik}}{\partial x^j} (\gamma^1,\cdots,\gamma^n) c^k \right)(w) = \\ 
- \frac{\left(\sum_{k,l=1}^n g_{ik} \Gamma^{k}_{jl} c^l \right)(w)}{z-w}, 
\label{eq:nasty1}
\end{multline}
which is zero only when the metric is flat!

These are the major complications in trying to understand the supersymmetry of the chiral de Rham complex on special holonomy manifolds: most of the physics literature is based upon expansions near a flat metric, in which one has a Boson-Fermion system with a constant metric $g_{ij}$ as in the previous section. One then finds embeddings of the algebras described in \ref{sec:examples} into these free Boson-Fermion systems. It is quite complicated to then pass to global sections of $\Omega^{ch}_M$ in the presence of a non-constant $g_{ij}$. 
\end{nolabel}

\begin{nolabel}\label{no:overcome}
An approach to deal with the problems mentioned in \ref{no:difficulty-in-cdr} was started in \cite{heluani2} and further developped in \cite{heluani8,heluaniholonomy}. The idea is to provide two embeddings of $T^*M \hookrightarrow \Omega^{ch}_M$ generalizing that of the (dual of the) previous section. So let $(M, g)$ be a manifold with a non-degenerate metric and let $\nabla$ be the Levi-Civita connection of $g$. Let $\Gamma_{ij}^k$ be the Christoffel-Symbols of $\nabla$ defined on local coordinates. 

Locally on a coordinate patch $(U, \left\{ x^i \right\}_{i=1}^n)$ we define. 
\begin{equation}\label{eq:local-e}
e^i_{\pm} = \frac{\sum_{j=1}^n g^{ij}b_j \pm c^i}{\sqrt{\pm 2}} \in C^{\infty}(U,\Omega^{ch}_M), \qquad i = 1,\cdots,n.
\end{equation}
These are the duals to the local sections of $\Omega^{ch}_M$ defined by the RHS of \eqref{eq:embedding2}. 

Let now $\omega \in C^\infty(M, \wedge^k T^*M)$ be a smooth $k$-form. Locally on the coordinate patch $U$ this form is written as 
\[ \sum_{i_1,\ldots,i_k} w_{i_1,\ldots,i_k} \, dx^{i_1} \wedge \cdots \wedge dx^{i_k}. \]

Define the numbers $T_{r,s}$ as the coefficients of the Bessel polynomials \cite{grosswald}
\begin{equation}
y_r  (x) = \sum_{s=0}^r T_{r,s} x^s = \sum_{s=0}^r \frac{(r + s)!}{(r-s)! s! 2^{s}} x^s,
\label{eq:bessel1}
\end{equation}
and let $T_{r,s}:=0$ when $s<0$ or $s>r$.  The following is the main technical result of \cite{heluaniholonomy}
\begin{thm*}
The local sections 
\begin{multline}\label{eq:local-currents}
J_\pm = \frac{1}{k!} \sum_{i_1,\dots,i_k}
\sum_{s=0}^{\lfloor\frac{k}{2}\rfloor} T_{k-s,s}  \sum_{\stackrel{j_1,\dots,j_{2s-1}}{l_1,\dots,l_{2s -1}}} 
\omega_{i_1,\dots,i_k}  \Gamma_{j_1,l_1}^{i_1} g^{i_2 j_1} \partial \gamma^{l_1} \dots \\  \dots    \Gamma_{j_{2s -1}, l_{2s-1}}^{i_{2s-1}} g^{i_{2s},j_{2s-1}} \partial \gamma^{l_{2s-1}} \Bigl(e_\pm^{i_{2s+1}} \Bigl( e_{\pm}^{i_{2s+2}} \Bigl( \dots e_{\pm}^{i_{n}} \Bigr) \Bigr) \dots \Bigr) \in C^{\infty}(U, \Omega^{ch}_M)
\end{multline}
Agree on intersections and hence they are restrictions to $U$ of two well defined global sections $J_\pm \in C^\infty(M, \Omega^{ch}_M)$. 
\end{thm*}
\end{nolabel}
\begin{nolabel}\label{no:examples}
The difficulty in finding global sections of $\Omega^{ch}_M$ lies in the non-associative nature of the normally ordered product. This is why we need to be careful with the parenthesis in expressions like \eqref{eq:local-currents}. To ilustrate the Theorem above let us analyse the first few examples. For $k=1$, we are given a one form $\omega = \sum_i \omega_i dx^i$ and we are associating the global sections 
\begin{equation}\label{eq:1-forms}
\sum_{i=1}^n \omega_i(\gamma^1, \dots,\gamma^n) e_\pm^i. 
\end{equation}

For $2$-forms $\omega = \sum_{ij} \omega_{ij} dx^i \wedge dx^j$ we already see a correcting term:
\begin{equation}\label{eq:2-forms}
J_\pm = \frac{1}{2} \sum_{i,j=1}^n \omega_{ij} \Bigl( e^i_{\pm} e^j_{\pm} \Bigr) + \frac{1}{2} \sum_{i,j,k,l = 1}^n \omega_{ij} \Gamma^{i}_{kl} g^{jk} \partial \gamma^l. 
\end{equation}
\end{nolabel}
Each time we increase the degree of the form $\omega$ by two we need to insert an extra correcting term of the form $\Gamma_{kl} g^{jk} \partial \gamma^l$. The geometric meaning of the map \eqref{eq:local-currents} is unclear to the author.
\begin{nolabel}\label{no:program}
Theorem \ref{no:overcome} provides with two embeddings 
\begin{equation}\label{eq:non-commuting}
\wedge^* T^*M \hookrightarrow \Omega^{ch}_M, \qquad \omega \mapsto J_\pm. 
\end{equation}
We have already seen that the two images $J_\pm$ of a one form $\omega \in T^*M$ commute, but that is not the case with their superpartners. It follows that the above embeddings do not provide with two commuting copies of the Boson-Fermion system inside of $\Omega^{ch}_M$. 

Something special happens when we have parallel forms on $M$, it turns out that the algebra generated by the images of these forms under the above two embeddings do commute inside $\Omega^{ch}_M$ and they turn out to be precisely the list of superconformal algebras of Section \ref{sec:examples}. 
\end{nolabel}

\section{Supersymmetric structures} \label{sec:susy-cases}
\begin{nolabel}\label{no:n=2-thm}
Let $M$ be a Calabi-Yau manifold, so we have a Ricci flat metric $g$, a complex structure $\mathcal{J}$ and a Kahler form $\omega \in C^\infty(M, \wedge^{1,1} T^*M) \subset C^{\infty}(M, \wedge^2 T^*M)$ which are compatible. The Kahler form $\omega$ is parallel $\nabla \omega = 0$. We have therefore two sections $J_\pm \in C^\infty(M, \Omega^{ch}_M)$ given by \eqref{eq:2-forms}. They commute:
\begin{equation} \label{eq:commuting1}  J_+(z) \cdot J_-(w) \sim 0. \end{equation}
We also have their superpartners $\ds J_\pm$. Define $G_\pm$ by
\begin{equation} \label{eq:single-pole} (\ds J_\pm)(z) \cdot J_\pm(w) \sim \frac{G_\pm(w)}{z-w}, \end{equation}
\begin{thm*}[\cite{heluani8}]
The vectors $J_\pm$ and $G_\pm$ generate two commuting copies of the $N=2$ superconformal algebra as in \ref{no:n=2} of central charge $c = 3/2 \dim_{\mathbb{R}} M$. 
\end{thm*}
Part of the theorem is to prove that in the OPE of \eqref{eq:single-pole} no higher order poles appear, therefore obtaining a well defined section $G_\pm$ by the RHS. The second point is to prove that in addition to \eqref{eq:commuting1} we have 
\begin{equation}\label{eq:commuting2}
(\ds J_\pm)(z) \cdot J_{\mp}(w) \sim 0. 
\end{equation}
This guarantees that the algebra generated by $(J_+, G_+)$ and their superpartners commutes with the algebra generated by $(J_-, G_-)$ and their superpartners. 

The point of this theorem is that we know the sections $J_\pm$ in any coordinate system and are given by the expression \eqref{eq:2-forms}. Of course if we look at special coordinate charts, for example by choosing holomorphic coordinates $\left\{ z^{\alpha} \right\}$ and antiholomorphic coordinates $\left\{ z^{\bar{\alpha}} \right\}$, we diagonalize the complex structure $\omega_{ij} g^{jk}$ appearing in the second term of \eqref{eq:2-forms}. This second term now reads in this holomorphic coordinates explicitly as:
\[ \frac{\sqrt{-1}}{2} \sum_{\alpha,\beta} \Gamma^\alpha_{\alpha \beta} \partial \gamma^\beta - \frac{\sqrt{-1}}{2} \sum_{\bar \alpha,\bar \beta} \Gamma^{\bar{\alpha}}_{\bar\alpha \bar\beta} \partial \gamma^{\bar \beta} = \frac{\sqrt{-1}}{2} \sum_{\alpha} \frac{\partial}{\partial z^\alpha} \log \det \sqrt{g} \partial \gamma^\alpha - \frac{\sqrt{-1}}{2} \sum_{\bar \alpha} \frac{\partial}{\partial z^{\bar\alpha}} \log \det \sqrt{g} \partial \gamma^{\bar \alpha}   \ \]
and since the manifold is Calabi-Yau we can cover $M$ with coordinate charts such that $\log \det \sqrt{g}$ is constant, that is, coordinate systems where the global holomorphic volume form looks like $dz^1 \wedge \dots \wedge dz^n$. In these coordinates we have that $J_\pm$ looks simply as the first term of \eqref{eq:2-forms} which in turn is 
\begin{equation} \label{eq:new-generators} J_\pm =  \frac{\sqrt{-1}}{2} \left( \sum_{\alpha} c^{\alpha} b_\alpha - \sum_{\bar\alpha}c^{\bar \alpha} b_{\bar \alpha} \pm \sum_{\alpha, \bar\beta} \left( g^{\alpha \bar \beta} b_{\alpha} b_{\bar\beta} + g_{\alpha \bar \beta} c^\alpha c^{\bar \beta} \right) \right).\end{equation}
\end{nolabel}
\begin{nolabel}
Notice that even using holomorphic coordinates such that the global holomorphic volume form of the Calabi-Yau manifold $M$ is constant, the generator \eqref{eq:new-generators} do not coincide with those of \eqref{eq:current2} coming from the identification \eqref{eq:tensor-iso}. 

The expressions for $G_\pm$ are not trivial, even in holomorphic coordinates they are cubic in the fermionic generators. We do not know of explicit expressions for $G_\pm$ on general coordinates, in the special coordinates that are holomorphic and such that the holomorphic volume form looks like $dz^1 \wedge\dots \wedge dz^n$ these sections are given by \cite{heluani8}:
\begin{equation}
		\label{eq:htildedef} 
		G_\pm = \frac{1}{2} \sum_{i} \left( b_i \partial \gamma^i + \beta_i c^i \right) \pm \frac{1}{2} \sum_{i,j,k,l} \left(
		\Gamma^k_{jl} g^{ij} c^l(b_i b_k) + 
		g^{ij} b_i \beta_j + g_{ij} c^i \partial \gamma^j \right)
\end{equation}
Note that these $G_\pm$ however are corrections to \eqref{eq:boson-fermion-gpm} involving the Christoffel symbols and a cubic term when the metric is not flat. Thus, these two commuting $N=2$ structures are indeed a curved manifestation of the structures obtained by two embeddings of a Boson-Fermion system inside of the $bc-\beta\gamma$ system and not the decomposition into holomorphic and anti-holomorphic commuting parts. 
\end{nolabel}
\begin{rem}
In \eqref{eq:htildedef} we see a typical pattern where we have local expressions for sections of $\Omega^{ch}_M$ that we only know in special coordinate systems. In fact at no point we needed to really know this expression for $G_\pm$. As mentioned in Remark \ref{rem:generating-currents}, all we need to produce the $N=2$ algebra is to construct the global section of $\Omega^{ch}_M$ that corresponds to $J$. Since $\Omega^{ch}_M$ is a sheaf of suspersymmetric algebras, we get for free another section, its superpartner $\ds J$. From the OPE of these two we get $G$ simply defined by the RHS of \eqref{eq:n=2-j-ope}. Since $\Omega^{ch}_M$ is a sheaf of vertex algebras, this section $G$ so defined is guaranteed to be a global section, being the OPE of two globally defined sections. Even though we do not know its local expression. 

The expression in \eqref{eq:htildedef} was obtained by a computation performed in holomorphic coordinates, where certain correcting term vanished due to a particular expression for the holomorphic volume form of $M$. This is why we only know that $G_\pm$ have this form on these coordinates, even though the above expressions make sense on any Riemannian manifold. This leads us to the first open problems on the topic:
\label{rem:g_pm_unknown}
\end{rem}
\begin{ques}
Is there a topological/geometrical invariant for the local sections $G_\pm$ defined by \eqref{eq:htildedef} to be globally defined sections of $\Omega^{ch}_M$?
\end{ques}
\begin{ques}
In case $G_\pm$ given by \eqref{eq:htildedef} are globally defined sections of $\Omega^{ch}_M$, do they generate two commuting copies of the $N=1$ superconformal algebra of central charge $c = 3/2 \dim_{\mathbb{R}}M$?
\end{ques}
All examples where it is known that $G_\pm$ are defined are orientable and spin manifolds. In these cases, the zero modes of $G^\pm$ can be identified with the corresponding Dirac operators on certain subspaces of forms as we will see below. 
\begin{nolabel}
In the case of $V(\fg_{super})$ constructed from a Courant algebroid over a point as in \ref{no:courant1}, the role of the Christoffel symbols is played by the structure constants of $\fg$ and the above construction for $G$ coincides with the Kac-Todorov construction.
\label{no:cubic-dirac}
\end{nolabel}
\begin{nolabel}\label{no:exemplos-dificil}
In principle an approach to answer the first question is straightforward: just perform a change of coordinates, apply \eqref{eq:change1}--\eqref{eq:change3} to each of the generators and then use quasi-associativity on the vertex algebra to collect terms. Any remaining term should give a topological/geometrical class on $M$ that should vanish for $G_\pm$ to be well defined. The second question is now simply a local computation so it is a matter of computing OPEs in a free field theory of two vectors explicitly given and there exists well established software packages to assist. 

This is essentially the approach taken in \cite{malikov} in the case of the \emph{diagonal} $N=2$. Indeed if we look at $G = G_+ + G_-$, after performing a change of coordinates $\left\{ x^i \right\} \mapsto \left\{ \tilde{x}^i \right\}$ and collecting the terms that appear from quasi-associativity in the transformation rule for \eqref{eq:change3} for the second term $\beta_i c^i$, we find that 
\[ \tilde{G} = G + \sum_i \frac{\partial}{\partial \gamma^i} \left( \mathrm{Tr} \log \left( \frac{\partial \tilde{x}^j}{\partial x^k} \right)_{jk} \right) c^i, \]
and we can recognize the first Chern class of the tangent bundle $TM$ in the RHS. 

The difficulty in trying to apply the same approach to these expressions \eqref{eq:htildedef} rely on the non-tensorial nature of the Christoffel symbols in as much as the cubic nature in the Fermions. 
Ricci-flatness was explicitly used in order to obtain the expressions \eqref{eq:htildedef}, hence it will not be a surprise if this is needed in the general case as well. 

Both of these questions exemplify the difficulties as well as the great help that having extra symmetries produce: in the $N=2$ case, the current $J$ was all we needed, and checking that the local expression for $J$, which is quadratic in the Fermions, is well defined amounts to the above computation of the Chern class. We get the existence of the fields $G$ for free. Another advantage of the supersymmetric approach is that at no point we need to compute the OPE $G(z) \cdot G(w)$.
\end{nolabel}
\begin{nolabel}\label{no:curved-topological-twist}
We proceed now to perform a topological twist in the curved situation. Let $M$ be a Calabi-Yau $2n$-manifold and consider the $C^\infty$ chiral de Rham complex $\Omega^{ch}_M$. It has two commuting $N=2$ structures of central charge $c=3n$ as in Theorem \ref{no:n=2-thm}. Let us fix a covering by special holomorphic coordinates $\left\{ z^\alpha \right\}$ as above. Notice $J_\pm$ can be written locally as 
\[ J_{\pm} =  \sum_{\bar\alpha} e_{\pm}^{\bar \alpha} e^{\pm}_{\bar \alpha} = -  \sum_{\alpha} e_\pm^\alpha e^\pm_\alpha = \]
where $e_{\pm}^{\alpha}, e_{\pm}^{\bar \alpha}$ are defined in \eqref{eq:local-e} and $e^\pm_\alpha, e^{\pm}_{\bar \alpha}$ are given by 
\begin{equation}\label{eq:locally-e2}
e^\pm_i = \frac{b_i \pm \sum_{j =1}^n g_{ij} c^j}{\sqrt{-2}}. 
\end{equation}
Of course there are relations among these generators and one can \emph{raise or lower} indices by contracting with the metric $g$ or its inverse. 

We can perform the topological twist as in \ref{no:topological-twist1} on one of the two sectors, say the ``minus'' sector and consider BRST cohomology. That is, we define $Q^-, H^-$ by 
\begin{equation}\label{eq:q-left}
\ds J_- = H^{-} - Q^-, \qquad G_- = Q^- + H^-.
\end{equation}	
These two odd generators together with $J_-$ and $T_- = L_- + \frac{1}{2} \partial J_-$ generate a topological vertex algebra. We can consider the cohomology $\mathcal{H}(\Omega^{ch}_M, Q^-_0)$, which will be concentrated in conformal weight $0$ with respect to $T_-$. Locally, all the generators $e^{\alpha}_+$, $e_{\alpha}^+$, their superpartners and derivatives commute with $T_-, J_-, Q_-$ and $H_-$. The generators $e^{\bar \alpha}_-$ have conformal weight zero and charge $+1$ while the generators $e_{\bar\alpha}^-$ have conformal weight $1$ and charge $-1$. It follows that the space of conformal weight zero for $T_-$ is generated by
\begin{enumerate}
\item all smooth functions $f(\gamma^1, \cdots,\gamma^n)$,  
\item polynomials in the fields $e^{\bar\alpha}_-$. 
\item all the ``plus'' generators $e^{\alpha}_+$, $e_{\alpha}^+$, their superpartners and their derivatives.
\end{enumerate}

Notice that a) and b) generate 
\begin{equation} \label{eq:embedding4} \wedge^* T^*_{0,1} \hookrightarrow \Omega^{ch}_M \end{equation} 
but this embedding is given in a different way than \eqref{eq:embedding1}. 
The differential $Q^-_0$ restricted to this space simply acts as $\bar\partial$, hence locally, the cohomology sheaf $\mathcal{H}(\Omega^{ch}_M, Q^-_0)$ is generated by holomorphic functions $f(\gamma^\alpha)$ and the fermions $e^\alpha_+$, $e_{\alpha}^+$. Noting that the OPE between these fields is simply given by
\begin{equation} \label{eq:1-1forms} e^\alpha_+(z) \cdot e^{\beta}(w) \sim \frac{g^{\alpha,\beta}(\gamma^1,\cdots,\gamma^n)(w)}{(z-w)} = 0, \qquad e^\alpha_+(z) \cdot e_\beta^+(w) \sim \frac{\delta^\alpha_\beta}{z-w}\end{equation}
We see that this cohomology is isomorphic to the holomorphic chiral de Rham complex $\Omega_M^{ch,hol}$. 
\end{nolabel}
\begin{nolabel}\label{no:kahler}
In the preceding section, to perform the topological twist and taking BRST cohomology we did not make use of the full $N=2$ superconformal structure, we simply needed the zero modes of the corresponding fields. These zero modes are well defined on any Kähler manifold, without the need for it to be Calabi-Yau\footnote{This is proved in the same way as in \cite{malikov} since the change of coordinates only involves derivatives of fields, it does not affect the zero modes.}. We have used the Kähler condition for example in \eqref{eq:1-1forms}. In fact, we arrive to the following
\begin{thm*}\cite{heluanimetric}
Let $M$ be a Kähler $2n$ manifold, and let $Q^-_0$ be the endomorphism of $\Omega^{ch}_M$ defined by \eqref{eq:q-left}. The cohomology sheaf
$\mathscr{H} (\Omega^{ch}_M, Q^{-}_0)$ is isomorphic to the holomorphic chiral de Rham complex $\Omega^{ch,hol}_M$ of $M$. 

If $M$ is Calabi-Yau, the $N=2$ superconformal structure generated by $J_+$, $G_+$ and their superpartners commute with $Q^-_0$ hence they survive in cohomology and they define an $N=2$ superconformal structure on $\Omega^{ch,hol}_M$ with central charge $c=3n$. 
\label{thm:kahler}
\end{thm*}

Notice that we have used special coordinate systems like holomorphic coordinates only to be able to compare this sheaf with the holomorphic chiral de Rham complex. On the other hand, the currents $J_\pm$ where defined in general only using the fact that we had the corresponding two-forms at hand. The BRST differential $Q^-_0$ is well defined without reference to holomorphic coordinates by \eqref{eq:q-left}. 
\end{nolabel}
\begin{nolabel}\label{no:generalized-kahler}
In fact one can perform the above procedure of \emph{twisting} half of the chiral de Rham complex and performing BRST cohomology on any generalized Kahler manifold, this provides a definition of the \emph{holomorphic} chiral de Rham complex for such manifolds \cite{heluanimetric}. 
\end{nolabel}
\begin{nolabel}\label{no:laplacian}
The restriction of the operators $Q^-_0$, $H^-_0$ and $T^-_0$ to the space of all differential forms act as the operators $\bar{\partial}$, $\bar \partial^*$ and $\Delta_{\bar\partial}$, the OPE of $Q^-$ with $H^-$ imply the Kähler identity
\[ [Q^-_0, H^-_0] = T^-_0 \leadsto [ \bar\partial, \bar\partial^* ] = \Delta_{\bar\partial}. \]
And therefore only harmonic forms  (a finite dimensional vector space if $M$ is compact Kähler) survive in cohomology. This can be generalized to show that in the cohomology $H(\Omega^{ch}_M, Q^-_0)$ we have finite dimensional energy spaces. 
\end{nolabel}
\begin{nolabel}\label{no:other-order}
We can take cohomologies in the reverse order, namely first take global sections of $\Omega^{ch}_M$ (the higher sheaf cohomologies vanish) and then proceed to take the cohomology with respect to $Q_0^-$. Since for each conformal weight the charge with respect to $J^-_0$ is bounded, standard arguments on spectral sequences show that this cohomology converges to the sheaf cohomology $H^*(M, \Omega^{ch,hol}_M)$. In the Calabi-Yau case, this vertex algebra acquires an extra $N=2$ structure from $J_+, G_+$ and their superpartners. 

We have achieved a solution to the problem described in \ref{no:drawback2}, we started from the $C^\infty$ chiral de Rham complex of $M$ and by performing a BRST cohomology we have obtained a vertex algebra that has finite dimensional conformal weight spaces. We have done this without making use of special coordinate systems on $M$.  
We now proceed to try to apply this approach to other types of manifolds. 
\end{nolabel}
\begin{nolabel}\label{no:n=4-sym}
Let $M$ now be a Hyper-Kahler $4n$-manifold. We have at our disposal three Kähler forms $\omega^i$, $i=1,2,3$, and a metric $g$. The corresponding three complex structures $\mathcal{J}^i$ satisfy 
\[ {\cJ^i}^2 = -\id, \qquad \cJ^i \cdot \cJ^j = \varepsilon_{ijk} \mathcal{J}^k. \]
\begin{thm*}\cite{heluani8}
Let $J^i_{\pm} \in C^{\infty}(M, \Omega^{ch}_M)$ be the six sections given by \eqref{eq:2-forms} corresponding to the Kähler forms $\omega^i$. Let $G^i_\pm$ be defined by \eqref{eq:single-pole} with $J^i_\pm$ in place of $J_\pm$. Then
\begin{enumerate}
\item $G_\pm:=G^1_\pm = G^2_\pm = G^3_\pm$. 
\item $(J^1_+, J^2_+, J^3_+, G_+)$ and $(J^1_-,J^2_-,J^3_-, G_-)$ and their superpartners generate two commuting copies of the $N=4$ superconformal vertex algebra with central charge $6n$. 
\end{enumerate}
\end{thm*}
\end{nolabel}
\begin{nolabel}\label{no:topological-twist-n=4}
In order to perform the topological twist in the Hyper-Kähler case, we choose any complex structure $\cJ$ in the sphere of complex structures of $M$ and we will obtain two corresponding $N=2$ structures $(J_\pm, G_\pm)$ inside of the above mentioned $N=4$ structures. Since the two $N=4$ structures commute, in particular the ``$+$''  structure will survive in the $Q_-^-$ BRST cohomology, therefore we obtain
\begin{thm*}
Let $\Omega^{ch,\cJ}_M$ be the holomorphic chiral de Rham complex of $M$ endowed with the complex structure $\cJ$. The cohomology $H^*(M, \Omega^{ch,\cJ}_M)$ admits an embedding of the $N=4$ superconformal algebra of central charge $c=3n$. 
\end{thm*}
\end{nolabel}
\begin{nolabel}\label{no:holomorphic-symplectic}
We have obtained the above theorem \ref{no:topological-twist-n=4} by performing a BRST reduction and identifying the resulting sheaf with the holomorphic chiral de Rham complex. In this case however we have at our disposal special coordinates since we have chosen a particular complex structure $\cJ$ making $M$ into a holomorphic-symplectic manifold. 

Hence let $(M, \omega)$ be a holomorphic symplectic $4n$-manifold. Locally the holomorphic symplectic form $\omega$ is given by $\omega = \sum \omega_{\alpha \beta} dx^{\alpha} \wedge dx^\beta$ and its inverse bivector by $\omega^{-1} = \sum \omega^{\alpha \beta} \partial_{x^\alpha} \partial_{x^{\beta}}$. Consider the local sections given by 
\begin{equation}  \label{eq:global-defined}
\begin{aligned}
G &= \sum_\alpha b_\alpha \partial \gamma^\alpha + \beta_\alpha c^{\alpha}, \\ 
E &= \sum_{\alpha \beta} \omega_{\alpha\beta} c^\alpha c^\beta, \\ 
F &= \sum_{\alpha,\beta} \omega^{\alpha \beta} b_\alpha b_\beta \\
J &= \sum_\alpha c^{\alpha} b_{\alpha}. 
\end{aligned}
\end{equation} 
\begin{thm*}
Then these four sections together with their superpartners give rise to well defined global sections $H^0(M, \Omega^{ch, hol}_M)$ that generate a copy of the $N=4$ vertex algebra of central charge $c = 3n$. 
\end{thm*}
\begin{proof}
The fact that $E$ and $F$ give rise to well defined sections is obvious on any manifold and corresponds to the embedding \eqref{eq:embedding1} and its dual. The fact that $J$ and $G$ are well defined sections giving rise to an $N=2$ superconformal structure of central charge $c=6n$ was proved in \cite{malikov} and uses the fact that $M$ is Calabi-Yau. We can perform the computation in Darboux coordinates where the $\omega_{\alpha\beta}$ and $\omega^{\alpha\beta}$ are constant, in which case the superpartners of the above fields are simply given by
\begin{equation}\label{eq:superpartnersn=4}
\begin{aligned}
\ds G = 2 L &= 2 \sum \beta_\alpha \partial\gamma^\alpha + \sum_\alpha b_\alpha \partial c^\alpha - \sum_\alpha c^\alpha  \partial b_\alpha, \\
\ds E &= - 2 \sum_{\alpha \beta} \omega_{\alpha \beta} c^\alpha \partial \gamma^\beta, \\
\ds F&= -2 \sum_{\alpha \beta} \omega^{\alpha\beta} b_\alpha \beta_\beta \\ 
\ds J&= \sum_\alpha b_\alpha \partial \gamma^\alpha - \sum_\alpha \beta_\alpha c^\alpha
\end{aligned}
\end{equation}
In these coordinates the computation is standard to check the OPE of the $N=4$ vertex algebra. 
\end{proof}
\end{nolabel}
\begin{nolabel}\label{no:remark-n=4}
As in the $N=2$ case, the expressions in \eqref{eq:global-defined}--\eqref{eq:superpartnersn=4} are holomorphic and therefore $C^\infty$-sections of $\Omega^{ch}_M$. Under this embedding \eqref{eq:embedding1} of the holomorphic chiral de Rham complex into the $C^\infty$-one they do not correspond to the $N=4$ structure described by Theorem \ref{no:topological-twist-n=4}. 
\end{nolabel}
\begin{nolabel}\label{no:bailin}
In the case of $M$ being a $K3$ surface, the above generators were used by B. Song to prove \cite{song1, song2} that the space of global sections $H^0(M, \Omega_M^{ch, hol})$ equals the irreducible quotient of the $N=4$ vertex algebra at central charge $c=6$. This also follows from the expansion of the elliptic genus as characters of the $N=4$ vertex algebra \cite{eguchi-hikami} by noting that the highest weights appearing in the decomposition do not appear in the Verma module for $\mathfrak{psl}(2|2)$ at level $-2$ \cite{kacwakimotomock2} (see section \ref{sec:hamiltonian}). 
\end{nolabel}
\begin{nolabel}\label{no:G2-lazaro}
Let now $M$ be a $G_2$ manifold. That is a smooth $7$-manifold with a metric $g$ with holonomy $G_2$. There is a three form $\phi \in C^\infty(M, \wedge^3 T^*M)$ that is covariantly constant with respect to the Levi-Civita connection. In fact this form determines the metric $g$ up to a conformal factor by 
\[ g(u,v) \wedge dvol_g = (u\llcorner \phi)\wedge (v \llcorner \phi) \wedge \phi \]
where $u,v$ are vector fields and $dvol_g$ is the volume form determined by $g$. The holonomy group of the metric $g$ thus determined by $\phi$ is included in $G_2$ if and only if $d \phi = 0$ and $d \star \phi = 0$. Note that this last equation is highly non-linear as the Hodge $\star$ operator depends on $g$ which in turn is determined by $\phi$. 

Theorem \ref{no:overcome} produces two sections $\Phi_\pm$ associated to $\phi$. The expressions for these sections is highly non-trivial as we have seen we need to include the correcting terms as in Theorem \ref{no:overcome}. In this situation we have Rodríguez's theorem 
\begin{thm}[\cite{lazaro-g2}]
The two sections $\Phi_\pm$ generate two commuting copies of the superconformal $SVG_2$ algebra of central charge $c = 21/2$ in $\Omega^{ch}_M$ in the sense of Remark \ref{rem:generating-currents}. 
\label{thm:lazaro}
\end{thm}
The proof of this theorem is a monumental and technically difficult computer-assisted tour de Force, where all the known algebraic and geometric identities between the Christoffel symbols, the metric and the three form $\phi$ on $M$ were needed. Including some fairly recent ones  \cite{corti}. 

A major difference between this situation and the Calabi-Yau/Hyper-Kahler situation is that there are no known \emph{good} coordinate systems adapted to $G_2$ manifolds like the holomorphic coordinates in the $N=2$ case. There is no analog of the local isomorphism \eqref{eq:tensor-iso}, hence we only have access to the local splitting into ``plus'' and ``minus'' Boson-Fermion systems as in \ref{no:boson-fermion-vs-ghosts}. 
\end{nolabel}
\begin{nolabel}\label{no:g2-topological-twist}
In the $G_2$ case, the vertex algebra of global sections $C^\infty(M, \Omega^{ch}_M)$ is too large, it has infinite dimensional conformal weight spaces. We would want to perform a half-twist as we did in the previous cases of $N=2$ and $N=4$ and obtain a vertex algebra with finite dimensional conformal spaces and possibly consider its character as an invariant of $M$, study its convergence and modular properties, etc. 

We have thus three problems at hand. The first problem is how to define a differential analogous to $Q^-_0$ such that its cohomology sheaf $\cH(\Omega^{ch}_M, Q^-_0)$ has finite dimensional conformal weight spaces? Analogously we would want that the cohomology with respect to $Q^-_0$ of the vertex algebra $C^{\infty}(M, \Omega^{ch}_M)$ have finite dimensional conformal weight spaces. 

The second problem is to study its character, this would be the analog of the elliptic genus for a $G_2$ manifold. 

The third problem and perhaps the most important in connection to Moonshine phenomena is to study the representation of $SVG_2$ in this vertex algebra. We would require that one of the two $SVG_2$ algebras produced by Rodriguez's theorem comutes with $Q^-_0$ so that it survives in cohomology. 

The first problem has been treated in the literature before and some progress has been done towards a topological twist in this setting. The main problem is that we do not have at our disposal any conformal weight $1$ fields to either change the Virasoro vector from central charge $c = 21/2$ to one of central charge $0$ (this was done by adding a multiple of $\partial J$ in the $N=2$ and $N=4$ cases) nor an odd vector of conformal weight $1$ with zero self-OPE to consider its zero mode as differential. It has been suggested in \cite{boer}, refining an idea of \cite{vafa}, to use the zero mode of an intertwining operator between Virasoro-modules as a BRST operator $Q^-_0$. 

The fields $\tfrac{-1}{5} X_\pm$ provide with two extra Virasoro fields of central charge $7/2$. It turns out that the original Virasoro fields $L_\pm = \tfrac{1}{2} \ds G_\pm$ decompose as sums of two commuting Virasoro fields $L_\pm = T_\pm + X_\pm$. We have thus four commuting Virasoro fields $(T_\pm, X_\pm)$. The same is not true for the corresponding $N=1$ structures which do not commute. The central charge of $T_\pm$ is $98/10 = 21/2 - 7/10$.   

The Virasoro algebra at central charge $c=7/10$ is one of the \emph{minimal models} for the Virasoro algebra. It is called the \emph{tri-critical Ising model} and it is a rational vertex operator algebra. Its list of simple modules is parametrized by the conformal weight of the cyclic vector. There are four irreducible modules with minimal conformal weights $0, \tfrac{1}{10}, \tfrac{6}{10}$ and $\tfrac{3}{2}$ and two Ramond-twisted irreducible modules with minimal weights $7/16$ and $3/80$. The fusion ring of this algebra is well known. As an example if we order the list of simple modules as above, the operation of tensoring with the module of minimal weight $1/10$ produces the sum of the module to the right plus the module to the left.  We now consider the sheaf of vertex operator algebras (or its global sections) $\Omega^{ch}_M$ as a $Vir_{7/2}\otimes Vir_{98/10}$-module (with respect to $(X_-, T_-)$). The vector $G_-$ has conformal weight $1/10$ with respect to $Vir_{7/10}$ as we can see from the third equation in \eqref{eq:svg2}. It follows that its field $G_-(z)$, when restricted to an irreducible $Vir_{7/10}$ submodule of $\Omega^{ch}_M$ will be the sum of two different intertwining operators $G_-(z) = H_-(z)+ Q_-(z)$, the first one intertwines to the left, and the second one to the right in the list of irreducibe $Vir_{7/10}$ modules. Taking the zero modes of $Q^{-}(z)$ we obtain an odd endomorphism of $\Omega^{ch}_M$ that squares to zero and such that the whole $SVG_2$ algebra generated by $\Phi^+$ commutes with it. It follows that $H(\Omega^{ch}_M, Q^-_0)$ is a vertex algebra that contains a $SVG_2$ subalgebra.  

The above discussion is only possible if the $Vir_{7/10}$ generated by $X$ is simple, that is, is a member of the minimal series. This leads us to 
\begin{ques}\label{ques:simpleg2}
Are the algebras $SVG_2 \subset \Omega^{ch}_M$ produced by Rodriguez theorem simple vertex algebras?
\end{ques}
The supersymmetry algebras $SVG_2$ can be produced by Hamiltonian reduction and as such they can be reduced from the universal affine Kac-Moody Lie algebras or their irreducible quotients. At this time we do not know if the copy of the $SVG_2$ algebra inside $\Omega^{ch}_M$ corresponds to the irreducible quotient or not. The question above, in case of the affirmative, would be an analog to Song's claim in the $N=4$ case of the K3 surface. It will also provide a unitary $Vir_{7/10}$ subalgebra generated by $X$ and hence we can apply the program of \cite{boer} to perform the topological twist. 

The restriction of $Q^-_0$ to the subspace of $\Omega^{ch}_M$ generated by $C^\infty(M)$ and the local one-forms $e^i_-$ is given as follows. One decomposes the space of all smooth forms $\wedge^* T^*M$ in irreducible representations of $G_2$, this decomposition is given by
\[ 
\begin{aligned}
\wedge^0 T^* M &= \Lambda^0_1, &&& \wedge^1 T^*M &= \Lambda^1_7, \\ 
\wedge^2 T^* M &= \Lambda^2_7 \oplus \Lambda^2_{14}, &&& \wedge^3 T^*_M &= \Lambda^3_1 \oplus \Lambda^3_7 \oplus \Lambda^3_{27}
\end{aligned}
\] 
where the subindex denotes the dimension of the irreducible $G_2$-module. The remaining spaces are obtained by applying the Hodge $\star$-involution, that is $\star \Lambda^n_m = \Lambda^{7-n}_{m}$. Restricting the standard de Rham complex by projecting to the $7$ and $1$ dimensional representations one obtains a complex:
\begin{equation} \label{eq:g2-derham}  0 \rightarrow \Lambda^0_1 \rightarrow \Lambda^1_7 \rightarrow \Lambda^2_7 \rightarrow \Lambda^3_1 \rightarrow 0. 
\end{equation} 
And an analogous one for higher forms. This complex coincides with the restriction of $Q^-_0$ to the space of differential forms. 

Unlike in the $N=2$ case \ref{no:laplacian}, we do not have at our disposition the operator $T^-_0$ as the zero mode of a Virasoro of central charge $0$. We however can consider the commutator of the zero modes $Q^-_0$, $H^-_0$ of the intertwining operators defined above. This operator, when restricted to the space of differential forms consists of a second order differential operator and coincides with the Laplacian (as one can check that $H^-_-$ coincides with $d^*$ in this case). Note that the embedding of forms into $\Omega^{ch}_M$ that we need to use is not given by \eqref{eq:embedding1} but rather the one obtained by Theorem \ref{no:overcome} in the ``minus''  sector. 

This in turn implies that the cohomology $H(\Omega^{ch}_M, Q^-_0)$ will have finite dimensional energy spaces.  It follows that if the answer to the question \ref{ques:simpleg2} is affirmative, we can attach the formal series
\begin{equation} \label{eq:g2formal}  \chi_M(\tau) = q^{-21/48} \mathrm{str}_{H(\Omega^{ch}_M, Q^-_0)} q^{L_0^+}, \qquad q = e^{2 \pi i  \tau}, \end{equation}
and ask if this series converges and has modular properties. 

Note that since we have two commuting Virasoros: $T^+, X^+$ we may promote the above function to a two variable trace by computing the joint eigenspaces of $T^+_0$ and $X^+_0$. 
\end{nolabel}

\begin{nolabel}\label{no:spin-7-symmetry}
Let us analyze now the case when the manifold $M$ has holonomy $Spin_7$. These are $8$ dimensional manifolds endowed with a four form $\psi$ which is parallel with respect to the Levi-Civita connection. In this case we conjectured 
\begin{conj*}[\cite{heluaniholonomy}] The vertex algebra $C^\infty(M, \Omega^{ch}_M)$ carries two global sections $X^\pm$ generating two commuting copies of  $SVSpin_7$ of central charge $c=12$ as in \ref{no:Spin7} in the sense of Remark \ref{rem:generating-currents}. 
\end{conj*}
This conjecture is verified in cases when the holonomy group of $M$ is properly included in $Spin_7$ this is for example the case of a $K3$-surface times $T^4$ (or $\mathbb{R}^4$ for a non-compact case), the product of two $K3$ surfaces, a Calabi-Yau $3$-fold  times $T^2$, a  $G_2$ manifold times an $S^1$, a Calabi-Yau four-fold, and of course when $M$ is flat (in which case it corresponds to the original work \cite{vafa}). 

The advantage we have in this situation over the $G_2$ case is that we have a well defined elliptic genus for $M$ \cite{hirzebruch-berger} which is a modular form for the congruence subgroup of $SL(2, \mathbb{Z})$ generated by $\tau \mapsto \tau +2$ and $\tau \mapsto -1/\tau$. In \cite{benjamin-harrison} (see also \cite{cheng-harrison}) a list of characters of unitary representations of $SPSpin_7$ was conjectured and a decomposition of the elliptic genus of $M$ as linear combinations of these characters was studied along the lines of \cite{eguchi} in the $K3$ case. The proposed characters of $SVSpin_7$ are Mock modular forms and therefore the generating function for the multiplicity spaces is a Mock modular form canonically attached to any $Spin_7$ manifold $M$. 

It should be possible to carry out rigourously this program within the context of the chiral de Rham complex of $M$ as in the $K3$ or Calabi-Yau case. The remainder of this section is conjectural and depends strongly on the existence of the two commuting copies of $SVSpin_7$ inside of $\Omega^{ch}_M$. 

The first step would be to prove the conjecture above. Assuming that we have two commuting copies of $SVSpin_7$ inside of $C^\infty(M, \Omega^{ch}_M)$ we want to obtain the elliptic genus of $M$ as some form of graded dimension of this vertex algebra. There are two possible approaches to this. On the one hand one could try to make sense of expressions as the ones appearing the physics literature like 
\begin{equation} \label{eq:elliptic-1} Z(\tau) = \tr_{C^{\infty}(M, \Omega_M^{ch})} (-1)^F q^{L_0^+ - c/24} \bar{q}^{L^-_0 - c/24}. \end{equation}
In the $Spin_7$ case assuming the conjecture we have at hand the two operators $L^\pm_0$. The global fermion number $F$ is also well defined as the zero mode of \eqref{eq:current} which is well defined on any orientable manifold. In fact, the zero mode is well defined on any manifold as noted in \cite{malikov}.  To decide whether the expression \eqref{eq:elliptic-1} defines a well defined function or not one needs to study the analytic properties of the operators $L^\pm_{0}$. In the Kähler case this can be formalized since both operators can be constructed as commutators of the fermions $Q^\pm$ and $H^\pm$ and by the K\"ahler identities we can identify $L^\pm_0$ with the Laplacian acting on tensor powers of the tangent and cotangent bundle as we did in \ref{no:laplacian}. One may use the ellipticity of these operators to prove the finite dimensionality of the joint energy spaces. 

Another approach which avoids analysis would be to perform the program suggested \cite{vafa} and in \cite{boer} for the $G_2$ case. It turns out that in $SVPsin_7$ the currents $\tfrac{1}{8}X$ is a Virasoro vector of central charge $c=1/2$. The algebra has two commuting Virasoro fields: $X$ of central charge $1/2$ and  $T=L - \frac{1}{8}X$ of central charge $23/2$. In our situation we would have four commuting Virasoro vectors $X^\pm, T^\pm$ in $C^\infty(M, \Omega^{ch}_M)$.

The Virasoro algebra at $c=1/2$ has three irreducible unitary representations, the lowest conformal weight being $0, \tfrac{1}{16}$ and $\tfrac{1}{2}$ respectively. That we consider as a list in this order. 

We therefore consider the sheaf (or its global sections) $\Omega^{ch}_M$ as a $Vir_{1/2} \otimes Vir_{23/2}$-module with respect to $X^-, T^-$. It follows from the second equation in \eqref{eq:svspin7} that the vector $G_-$ has conformal weight $1/16$ with respect to $X^-$. The corresponding operator $G_-(z)$ is an intertwining operator between the different $Vir_{1/2}$-modules appearing in the decomposition of $C^\infty(M, \Omega_M^{ch})$. The fusion rules in this situation are similar and simpler than in the $G_2$ case, the operation of tensoring with the module of minmial weight $1/16$ produces a sum of the module to the right plus the module to the left in the above list. It follows in the same way as in the $G_2$ case above that the intertwining operator $G_-(z)$, when restricted to an irreducible $Vir_{1/2}$-submodule of $\Omega_M^{ch}$ will be the sum of two different intertwining operators $G_-(z) = H_-(z) + G_-(z)$, the first one intertwining to the left and the second to the right in the above list of modules. Taking the zero mode of $Q^-(z)$ we obtain an odd endomorphism of $\Omega^{ch}_M$ that squares to zero and that the whole $SVSpin_7$ algebra generated by $X^+$ commutes with it. It follows that $H(\Omega^{ch}_M, Q_0^-)$ is a vertex algebra that contains one copy of $SVSpin_7$ as a subalgebra. 

Conjecturally one would obtain a vertex algebra such that its graded dimension:
\[ \tr_{H(\Omega^{ch}_M, Q^-_0)} q^{L_0^+ - c/24}, \]
coincides with the elliptic genus of $M$, in particular, it should have finite dimensional conformal weight spaces. 

Of course the above topological twist in the $Spin_7$ case would only work if the vertex algebras produced by the conjecture are irreducible, so that the copy of $Vir_{1/2}$ is a unitary vertex operator algebra and $\Omega_M^{ch}$ decomposes as a sum of irreducible modules. 

As in the $G_2$ case we can restrict the action of the endomorphism $Q^-_0$ to the space generated by $\gamma^i$ and the local 1-forms $e^i_-$. The space of $k$ forms $\wedge^k T^*_M$ decomposes under the action of $Spin_7$ as an orthonormal sum of irreducible representations $\Lambda^k_{l}$ of dimension $l$:
\[
\begin{aligned}
\wedge^1 T^*M &= \Lambda^1_8, && \wedge^2 T^*M &=& \Lambda^2_7 \oplus \Lambda^2_{21} \\ 
\wedge^3 T^*M &= \Lambda_8^3 \oplus \Lambda_{48}^3, && \wedge^4 T^*M &=& \Bigl(\Lambda_1^4 \oplus \Lambda_7^4 \oplus \Lambda^4_{27}\Bigr) \oplus \Lambda^4_{35}
\end{aligned}
\] 
and the Hodge $\star$-operation is an isometry $\Lambda^k_l \simeq \Lambda^{8-k}_l$. The summand $\Lambda^4_{35}$ is the $-1$ eigenspace of $\star$ in $\wedge^4 T*M$ while the rest is the $+1$ eigenspace. The form $\psi$ is a basis for the $\Lambda^4_1$ term.  The restriction of the operator $Q^-_0$ to $\wedge^\bullet T^*M$ is identified with the composition of the de Rham differential with the corresponding projections in the complex:
\[ 0 \rightarrow \Lambda^0_1 \rightarrow \Lambda^1_8 \rightarrow \Lambda^2_7 \rightarrow \Lambda^3_8 \rightarrow \Lambda^4_1 \rightarrow 0, \]
which in the case of $Hol(M)=Spin_7$ is a resolution of $\Lambda^0_1$ \cite{joyce}. 
\end{nolabel}
\section{Hamiltonian reduction and mock modular forms} \label{sec:hamiltonian}
\begin{nolabel}\label{no:w-algebras}
It is remarkable that all of the supersymmetric algebras appearing as symmetries of the chiral de Rham complex are obtained as quantum Hamiltonian reductions $W_k(\fg,f)$ of affine Kac-Moody algebras associated to finite dimensional Lie super-algebras $\fg$ and a nilpotent element $f \in \fg$. For an introduction  to these $W$-algebras and their representation theory we refer the reader to \cite{kacwakimoto1} and references therein. 
\end{nolabel}
\begin{nolabel}
Roughly speaking $W_k(\fg, f)$ is constructed as follows. One starts with a simple finite dimensional Lie superalgebra $\fg$ and an even nilpotent element $f \in \fg$ and a non-degenerate invariant bilinear form $(,)$ on $\fg$ normalized so that the highest root $\theta$ satisfies $(\theta, \theta)=2$. Extend the nilpotent element to an $\mathfrak{sl}_2$ triple $(f,h,e)$. The adjoint action of the semisimple element $h$ decomposes
\[ \fg = \bigoplus_{j \in \frac{1}{2} \mathbb{Z}} \fg_j, \qquad \fn_\pm = \bigoplus_{\pm j \geq 1/2} \fg_j,  \]
We normalize $h$ so that $f \in \fg_{-1}$. We construct three different vertex algebras associated to this setup:
\begin{enumerate}
\item $V_k(\fg)$ is the affine Kac-Moody vertex algebra \ref{no:vg} associated to $\fg$. 
\item $F(\fn_- \oplus \fn_+)$ is the free Fermions \ref{no:free-fermions}  associated to the vector space $\fn_- \oplus \fn_+$ with its symmetric non-degenerate form $(,)$. 
\item $F_{ne}(\fg_{1/2})$ is the symplectic Bosons \ref{no:symplectic-bosons} associated to the vector space $\fg_{1/2}$ and its symplectic form $(f,[\cdot,\cdot])$. 
\end{enumerate}
Consider the vertex algebra $C_k(\fg, f) = V_k(\fg) \otimes F(\fn_- \oplus \fn_+) \otimes F_{ne}(\fg_{1/2})$ and let 
\[ Q = \sum_\alpha (-1)^\alpha u_\alpha \varphi^\alpha - \frac{1}{2} \sum_{\alpha,\beta,\gamma} (-1)^{\alpha \beta} c_{\alpha \beta}^\gamma \varphi_\gamma \varphi^\alpha \varphi^\beta  + \sum_\alpha (f, u_\alpha) \varphi^\alpha + \sum_\alpha \varphi^\alpha \Phi_\alpha,\]
Where $\alpha, \beta, \gamma$ index a basis of $\fn_+$, $u_\alpha \in \fg_\alpha \subset V_k(\fg)$, $\varphi^\alpha$ is the same element considered as an element of $\fn_+ \subset F(\fn_+ \oplus \fn_-)$, $\varphi_\alpha$ is its dual element with respect to $(,)$ in $\fn_- \subset F(\fn_+ \oplus \fn_-)$, $\Phi_\alpha$ form a basis for $F_{ne}(\fg_{1/2})$ when the corresponding $u_\alpha \in \fg_{1/2}$ or zero otherwise. $(-1)^\alpha$ denotes $+1$ if $u_\alpha$ is even or $-1$ otherwise. Finally $c_{\alpha\beta}^\gamma$ are the structure constants of $\fn_+$. The vertex algebra $F(\fn_+ \oplus \fn_-)$ is $\mathbb{Z}$-graded by the eigenvalues of 
\[ J = \sum_\alpha \varphi^\alpha \varphi_\alpha, \]
called \emph{charge}. We will consider this grading in $C^\bullet_k(\fg,f)$. $Q$ has charge $1$ hence its zero mode $Q_0: C^{\bullet} \rightarrow C^{\bullet +1}$. 

We have the following
\begin{thm*}[\cite{kacwakimoto1}] In the situation described in above we have
\begin{enumerate}
\item The odd vector $Q$ satisfies $Q(z)\cdot Q(w) \sim 0$. 
\item $H^i(C^\bullet, Q_0) = 0$ if $i \neq 0$. 
\item Define the vertex algebra $W_k(\fg, f) := H^0(C^\bullet, Q_0)$. It is a conformal vertex algebra. For each vector $a \in \fg_{j}$ such that $[a,f]= 0$ there exists a vector in $W_k(\fg,f)$, primary of conformal weight $k+1$. $W_k(\fg,f)$ is strongly generated by such vectors corresponding to a basis of the centralizer $Z_\fg(f)$ of $f$ in $\fg$. 
\end{enumerate}
\end{thm*}
\end{nolabel}
\begin{nolabel}\label{no:free-field-realization}
Projecting the generating vectors of $W_k(\fg, f)$ to $V(\fg_0) \otimes F_{ne}(\fg_{1/2})$ gives a \emph{free-field} realization of the $W_k(\fg,f)$ algebra. Of course in order to obtain a true free-field realization one needs to consider a free field realization of $V(\fg_0)$, for example using the Wakimoto realization.
\end{nolabel}
\begin{nolabel}\label{no:examples-of-nilpotents}
When the nilpotent $f$ corresponds to the minimal root $-\theta$, the corresponding grading of $\fg$ is simply 
\[ \fg = k \cdot f \oplus \fg_{-1/2} \oplus \fg_0 \oplus \fg_{1/2} \oplus k \cdot e.\]
The vertex algebra $W_k(\fg, f)$ is generated by vectors of conformal weight $1$, $3/2$ and the Virasoro vector $L$ corresponding to $f$.

The other extreme is when the nilpotent element $f$ is the principal nilpotent $f = \sum e_{-\alpha_i}$ given as a sum of root vectors for the set of even primitive roots of $\fg$. In this case the corresponding $W$ algebra is called principal and was defined and studied in detail (in the non-super case) by Feigin and Frenkel in \cite{feigin-frenkel}. In this case $\fg_{1/2} = 0$, the grading of $\fg$ is integral and the free field realization is in $V(\fh)$, the algebra of Free Bosons based on the Cartan subalgebra $\fh \subset \fg$. 

In the super-Lie algebra case there is another interesting example which corresponds to the \emph{super-principal} nilpotent element $f$. This is the case when $\fg$ admits a set of simple roots $\alpha_i$ which are all odd, and considering $F = \sum e_{-\alpha_i}$, $E = \sum_{\alpha_i}$, $f = [F,F]$, $e = [E,E]$ and $h = [e,f]$ the Lie superalgebra $(f,F,h,E,e)$ is isomorphic to $\mathfrak{osp}(2|1)$. In this case the algebra $\fg_0 = \fh$ is the Cartan subalgebra of $\fg$ and $V(\fh) \otimes F_{ne}(\fg_{1/2})$ is a Boson-Fermion system as in \ref{no:boson-fermion-vs-ghosts}. In these cases the corresponding $W_{k}(\fg,f)$ algebra is supersymmetric and inherits the $N=1$ structure from the Boson-Fermion system.  
\end{nolabel}
\begin{nolabel}\label{no:list}
All of the supersymmetry algebras of the chiral de Rham complex discussed in the previous section are $W$ algebras. Most of them are superprincipal $W$-algebras.  They correspond to the entries in the following table:
\begin{center}
\renewcommand{\arraystretch}{1.5}
\begin{tabular}{c|c|c|c}
$\fg$ & $W(\fg,f)$ & $k$ & $Hol(M)$  \\ 
\hline
$\mathfrak{sl}(2|1)$ & $N=2$ & $\frac{-1-n}{2}$ & $SU(n)$ \\ 
$\mathfrak{psl}(2|2)$ & $N=4$ & $-1 -n$ & $SP(n)$\\
$\mathfrak{osp}(3|2)$ & $SVSpin_7$ & $\frac{1}{3}$ & $Spin(7)$ \\ 
$\mathfrak{osp}(4|2)$ & $SVG_2$ & $\frac{1}{3}$ &  $G_2$
\end{tabular}
\end{center}
In the $\mathfrak{sl}(2|1)$ case, namely $N=2$ supersymmetry, the superprincipal nilpotent agrees with the minimal nilpotent. The $N=4$ case it corresponds to the minimal nilpotent, the remaining cases are superprincipal. In the case of the minimal nilpotent one has a free-field realization from \cite{kacwakimoto1}, in the $Spin_7$ and $G_2$ cases free field realizations of the corresponding algebras have been computed in \cite{lazaro2, lazaro3}. 

It is natural to ask
\begin{ques*}
Can one construct the supersymmetries of the chiral de Rham complex directly by Hamiltonian reduction, namely, is there a morphism from $C^\bullet(\fg,f)$ to either $\Omega_M^{ch}$ or perhaps a resolution, inducing the embeddings described in the previous section?
\end{ques*}
An answer to this question would provide further relations between the representation theory of $\mathfrak{osp}(4|2)$ ( resp. $\mathfrak{osp}(3|2)$) at level $k=1/3$ and the geometry of $G_2$ (resp. $Spin_7$) manifolds. 
\end{nolabel}
\begin{nolabel}\label{no:characters}
In the minimal nilpotent cases the representation theory has been largely studied. The characters of unitary representations and their modularity properties are well known \cite{boucher-friedan-kent, eguchi-taormina1}. Recently the representation theory of these algebras has caught attention in connection with the Quantum Hamiltonian reduction of super Lie algebras described above \cite{kacwakimotomock,kacwakimotomock2, kacwakimoto1, kacwakimoto-superconformal}. 
Characters of (integrable) irreducible representations of affine Kac-Moody algebras form a vector valued modular form under the group $SL(2, \mathbb{Z})$. This is no longer the case for Lie superalgebras. In the supercase, in many studied situations, the characters form a vector valued \emph{mock-modular} form \cite{zwegers}. This has been known for many years in the case of $\mathfrak{sl}(2|1)$ \cite{kacwakimotoappel} and has been extended recently and systematically to different superalgebras by obtaining complete descriptions of characters of irreducible modules in the finite dimensional case \cite{cheng-duality} and in the affine case \cite{kac-gorelik1, kac-gorelik2, kacwakimotomock, kacwakimotomock2}. This was used to show that the characters of the irreducible (and more generally admissible) representations of the corresponding $W$ algebras provide examples of vector valued Mock modular forms. 

The works mentioned above concentrate on the minimal nilpotent case and this explains the mock modular properties of the characters of $N=2$ and $N=4$ modules. In the super-principal nilpotent case the list of characters has not been explicitly produced, however the techniques and the machinery of \cite{kacwakimoto-superconformal} is directly applicable. One would expect to confirm the list of characters of \cite{benjamin-harrison} in the $Spin_7$ case by studying the representation theory of $\mathfrak{osp}(3|2)$ at level $k=1/3$ and obtain new characters for the $G_2$ case by studying the representation theory of $\mathfrak{osp}(4|2)$ at level $k=1/3$. 
\end{nolabel}
\begin{nolabel}\label{no:moonshine}
The possible connection of Mathieu's group $M_{24}$ with the geometry of $K_3$ in the context of moonshine phenomena was discovered in \cite{eguchi}. In that work the authors expand the elliptic genus of a $K3$ surface as a linear combination of characters of the $N=4$ algebra. The elliptic genus of a Calabi-Yau manifold $M$ is the character of the sheaf cohomology of the holomorphic chiral de Rham complex $\Omega^{ch,hol}_M$ of $M$. We have seen in \eqref{eq:ell3}  that this can also be computed using the smooth chiral de Rham complex $\Omega_M^{ch}$, the resulting expansion is \cite{eguchi-hikami}
\begin{equation}\label{eq:ell5}
y^{- \frac{3}{4} \dim_{\mathbb{R}}M} \tr_{H \left( C^\infty(M, \Omega_M^{ch}), Q_0^- \right)} q^{T_0^+} y^{J^+_0} = 20 \ch_{\frac{1}{4},0}(\tau,\alpha) - 2 \ch_{\frac{1}{4},0} (\tau,\alpha) + \sum_{n \geq 0} A_n \ch_{n+ \frac{1}{4}, \frac{1}{2}}(\tau,\alpha), 
\end{equation}
Where $\ch_{h,l}(\tau,\alpha)$ denotes the character of the Ramond twisted irreducible representation of the $N=4$ algebra of central charge $c=6$, obtained as Hamiltonian reduction of $\mathfrak{psl}(2|2)$ at level $k=-2$, with highest weight vector of conformal weight $h$ and isospin $l$. The numbers $A_n$ are positive integers for all values of $n$ and in fact are dimensions of representations of $M_{24}$ \cite{gannon}. The generating series for these multiplicity spaces is a mock modular form. Indeed, the LHS of \eqref{eq:ell5} is a Jacobi form of weight $0$ and index $1$ while the characters $\ch_{h,l}(\tau,\alpha)$ are mock modular. 
\end{nolabel}

\begin{nolabel}\label{no:other-hol}
This prompts the question of whether one can associate a mock modular form canonically to other manifolds with special holonomy. In the case of manifolds with holonomy $Spin_7$ a program was  initiated in \cite{benjamin-harrison}. From what it was discussed above it would be interesting to prove:
\begin{ques*}
Is the elliptic genus of a $Spin_7$ manifold $M$ the graded dimension of the topological twist $H\left( C^{\infty}(M, \Omega^{ch}_M), Q^-_{0} \right)$ described in \ref{no:spin-7-symmetry}?
\end{ques*}
\begin{ques*}Are the conjectural characters of \cite{benjamin-harrison} obtained by quantum Hamiltonian reduction of characters of $\mathfrak{osp}(3|2)$ at level $k=1/3$ as described in \cite{kacwakimoto-superconformal}?
\end{ques*}
\begin{ques*}Is the decomposition of $\mathcal{Ell}_M(\tau, \alpha)$ in characters of $SVSpin_7$ coming from the decomposition of $H\left( C^{\infty}(M, \Omega^{ch}_M), Q^-_{0} \right)$ into irreducible representations of the $SVSpin_7$ from Conjecture \ref{no:spin-7-symmetry}?
\end{ques*}
At this time, the author does not know of any method that is not a lengthy direct computation in order to prove Conjecture \ref{no:spin-7-symmetry}. Given that conjecture, the topological twist described in \ref{no:spin-7-symmetry} and the machinery of \cite{kacwakimoto-superconformal} should be immediately applicable. 
\end{nolabel}
\begin{nolabel}\label{no:G2-final}
As we have already discussed the $G_2$ case is more subtle since we do not have a well defined elliptic genus. Rather in this case if the answer to Question \ref{ques:simpleg2} is affirmative we would consider \eqref{eq:g2formal} as a definition. By Rodriguez theorem this algebra will decompose as irreducible representations of $SVG_2$. Since one would expect the characters of $SVG_2$ to be mock modular, being obtained by Hamiltonian reduction of $\mathfrak{osp}(4|2)$ at level $k=1/3$ and the character \eqref{eq:g2formal} to be modular, we expect to attach the generating series of multiplicity spaces, a mock modular form, to any $G_2$ manifold. 
\end{nolabel}

\begin{nolabel}\label{no:conclude}
We conclude by mentioning other situation in which $W_k(\fg, f)$ algebras, at the minimal nilpotent make a striking appearance in connection to holonomy groups. Let $V$ be a vertex super algebras generated in conformal weights $1$, $3/2$ and a Virasoro vector in conformal weight $2$. The space $\fh$ of conformal weight $1$ is a Lie algebra with an invariant bilinear form $(,)$ and the space $M = V_{3/2}$ of conformal weight $3/2$ is a $\fh$-module. If one classifies the simple such vertex algebras one finds \cite{fradkin} three infinite classical series a continuous family and two exceptional cases. The pairs $\fh, M$ that appear are in correspondence with the (complex analogs of ) pairs $(hol(M), T^*_xM)$ of of holonomy Lie algebras and their irreducible representations on a cotangent fiber $T^*_xM$ for $x \in M$.  In fact, by requiring that $V$ has a quasi-classical limit one finds \cite{de-sole-thesis} that the pairs are such that the group $H \subset SO(M)$ associated to $\fh$ has an open orbit in the unit quadric $\{(m,m) = 1\} \subset M$, obtaining immediately the complex analog of Berger's list\footnote{Recall that Berger's classification of possible holonomy groups is also obtained by reducing the problem to groups having an open orbit on the sphere.}. All of these algebras are also quantum Hamiltonian reductions of Lie superalgebras $\fg$, but now at the minimal nilpotent element. Specializing to the same four holonomy groups that we treated in this article this list looks as follows (see also \cite{kacwakimoto1}, recall we are using complex Lie algebras)

\begin{center}
\renewcommand{\arraystretch}{1.5}
\begin{tabular}{c|c|c}
$\fg$ & $\fh$ & $M$  \\ 
\hline
$\mathfrak{sl}(2|m)$ & $\mathfrak{gl}_m$ & $\mathbb{C}^m \oplus \mathbb{C}^{m*}$ \\ 
$\mathfrak{osp}(4|m)$ & $\mathfrak{sl}_2 \oplus \mathfrak{sp_m}$ & $\mathbb{C}^2 \otimes \mathbb{C}^m$ \\ 
$F(4)$ & $Spin_7$ & $8-dim$ \\
$G(3)$ & $G_2$ & $7-dim$ \\ 
\end{tabular}
\end{center}
It is natural to ask:
\begin{ques*}
What is the connection betweeen these $W$ algebras at the minimal nilpotent and those in \ref{no:list}?
\end{ques*}
\end{nolabel}

\bibliographystyle{plain}
\def\cprime{$'$}

\end{document}